\documentclass[preprint,number,sort&compress]{elsarticle}

\usepackage{amsfonts}
\usepackage{amsmath}
\usepackage{amssymb}
\usepackage{graphicx}
\usepackage[vlined,boxruled]{algorithm2e}
\usepackage{color}
\usepackage{ifthen}
\usepackage{xspace}

\journal{Journal of Computational and Applied Mathematics}
\newcommand{\mathC}{\mathbb{C}}

\newcommand{\diag}{{\rm diag}}

\newcommand{\mat}{\textrm{mat}}
\newcommand{\trace}{\textrm{tr}}
\newcommand{\one}{\mathbf{1}}
\newcommand{\sumv}{\textrm{sum}}

\newcommand{\bicg}{BiCG\xspace}
\newcommand{\glbicg}{Gl-BiCG\xspace}
\newcommand{\eglbicg}{eGl-BiCG\xspace}
\newcommand{\libicg}{Li-BiCG\xspace}
\newcommand{\blbicg}{Bl-BiCG\xspace}
\newcommand{\blbicgrq}{Bl-BiCG-rQ\xspace}

\newcommand{\bicgstab}{BiCGStab\xspace}
\newcommand{\glbicgstab}{Gl-BiCGStab\xspace}
\newcommand{\libicgstab}{Li-BiCGStab\xspace}
\newcommand{\blbicgstab}{Bl-BiCGStab\xspace}
\newcommand{\blbicgstabrq}{Bl-BiCGStab-rQ\xspace}

\newcommand{\qmr}{QMR\xspace}
\newcommand{\glqmr}{Gl-QMR\xspace}
\newcommand{\eglqmr}{eGl-QMR\xspace}
\newcommand{\liqmr}{Li-QMR\xspace}
\newcommand{\blqmr}{Bl-QMR\xspace}

\newcommand{\blcg}{Bl-CG\xspace}

\newcommand{\gmres}{GMRES\xspace}
\newcommand{\blgmres}{Bl-GMRES\xspace}

\newdefinition{rem}{Remark}

\newtheorem{example}{Example}
\newtheorem{proposition}{Proposition}
\newproof{proof}{Proof}
\let\oldexample\example
\renewcommand{\example}{\oldexample\normalfont}

\ifthenelse{\isundefined{\dontprintsemicolon}}
{\newcommand{\NoSemicolon}{\DontPrintSemicolon}}
{\newcommand{\NoSemicolon}{\dontprintsemicolon}}

\graphicspath{{fig/}}

\begin{document}
\begin{frontmatter}

\title{On short recurrence Krylov type methods for linear systems with many right-hand sides}

\author[srge]{Somaiyeh Rashedi}
\ead{s\_rashedi\_t@yahoo.com}

\author[sbaf]{Sebastian Birk}
\ead{birk@math.uni-wuppertal.de}

\author[sbaf]{Andreas Frommer}
\ead{frommer@math.uni-wuppertal.de}

\author[srge]{Ghodrat Ebadi}
\ead{ghodrat\_ebadi@yahoo.com}

\address[srge]{Department of Mathematical Sciences, University of Tabriz, 51666-14766 Tabriz, Iran}
\address[sbaf]{Fachbereich Mathematik und Naturwissenschaften, Bergische Universit\"at Wuppertal, 42097 Wuppertal, Germany}

\begin{abstract}
Block and global Krylov subspace methods have been proposed as methods adapted to the situation where one
iteratively solves systems with the same matrix and several right hand sides. These methods are advantageous, since they
allow to cast the major part of the arithmetic in terms of matrix-block vector products, and since, in the block case, they
take their iterates from a potentially richer subspace. In this paper we consider the most established Krylov
subspace methods which rely on short recurrencies, i.e.\ BiCG, QMR and BiCGStab. We propose modifications of
their block variants which increase numerical stability, thus at least partly curing a problem previously observed by several authors.
Moreover, we develop modifications of the ``global'' variants which almost halve the number of
matrix-vector multiplications. We present a discussion as well as numerical evidence which both indicate that the additional
work present in the block methods can be substantial, and that the new ``economic'' versions of the ``global'' BiCG and QMR method can be considered as
good alternatives to the BiCGStab variants.
\end{abstract}


\end{frontmatter}

\section{Introduction}
We consider simultaneous linear systems with the same matrix $A \in
\mathC^{n \times n}$ and $s$ right-hand sides (r.h.s.) $b_i$,
\begin{equation} \label{eq:simultaneous_systems}
A x_i = b_i, \; i=1,\ldots,s,
\end{equation}
which we also formulate in block form as the matrix equation
\[
AX = B, \; \mbox{ where } X = [x_1|\ldots|x_s], B = [b_1|\ldots|b_s]
\in \mathC^{n \times s}.
\]
As a general rule we will use lower case letters with sub-indices
for the columns of the matrix denoted with the corresponding upper
case letter.

Our overall assumption is that $A$ is large and sparse and that a
(preconditioned) Krylov subspace type iterative method is to be used
to iteratively approximate the solutions for the different r.h.s.
Our interest is in methods which efficiently take advantage of the
fact that we are given the $s$ linear systems with the r.h.s.\
$b_i,i=1,\ldots,s$ simultaneously, and we use the generic name {\em
simultaneous methods} for any such method. Investigations in this
setting are not new, the development and analysis of {\em block
variants} of the standard Krylov subspace methods dating back to at
least the mid 1980s \cite{Oleary80}. With the advent of massively
parallel processor architectures such as GPUs, where memory access
is usually the determining factor for computational speed,
simultaneous linear systems of the form
\eqref{eq:simultaneous_systems} offer the possibility of optimizing
memory access to $A$ by reading and storing $A$ only once while
updating the iterates for all $s$ r.h.s.\ simultaneously. Algorithms
which take this into account can therefore be expected to run faster
because their arithmetic operations perform faster. Moreover,
simultaneous methods also offer the potential to, in addition,
increase the speed of convergence when obtaining their iterates from
a larger, ``block'' Krylov subspace rather than the standard Krylov
subspaces.

As a rule of thumb, working with $s$ several r.h.s.\ simultaneously
instead of just one increases the number of vectors to be stored by
a factor of $s$. This is why simultaneous variants of the \gmres
method may, due to the long recurrences present in \gmres, suffer
from the fact that the method has to be restarted after a relatively
small number of iterations. This typically degrades the speed of
convergence.

In this paper we therefore focus on simultaneous Krylov type methods
relying on short recurrences. We assume that $A$ is a general
non-singular matrix, i.e.\ we do not assume that we can take
advantage of additional properties of $A$ like symmetry, e.g.
Typical and well-established methods for a single right-hand side in this
case are \bicg (bi-conjugate gradients \cite{Fletcher74}), \bicgstab
(stabilized \bicg \cite{vdV92}) and \qmr (quasi-minimal residual
method \cite{FrNa92}), and we mainly develop on three different
simultaneous variants for each of these three methods in this paper.
In section~\ref{sec:types} we introduce the three different
principles termed {\em loop-interchanging}, {\em global} and {\em
block} to generate a simultaneous variant from a given (single
r.h.s.) Krylov subspace method and, as an example, describe the
resulting \bicg variants in some detail. We then devote
section~\ref{sec:literature} to a discussion of the various
simultaneous methods known from the literature in our context
before developing two improvements in
section~\ref{sec:improvements}. One improvement reduces the
arithmetic cost in the {\em global} \bicg and \qmr variants by nearly
a factor of 2, suppressing almost all of the matrix-vector
multiplications with $A^H$. The other improvement enhances numerical
stability in the {\em block} methods through the judicious use of
QR-factorizations of the matrix of simultaneous residuals.
Section~\ref{sec:experiments} contains an extensive numerical study
for the various methods and conclusions on which methods
are to be preferred in which situations.

\section{Loop-interchanging, global and block Krylov methods} \label{sec:types}

In this section we describe three different approaches to obtain a
simultaneous Krylov subspace method for $s$ r.h.s.\ from a given
method for a single r.h.s. All resulting simultaneous methods
reduce to the single r.h.s.\ method for $s=1$.

Solving the systems \eqref{eq:simultaneous_systems} one after the
other with the same Krylov subspace method can be viewed as
performing two nested loops with the outer loop (``\textbf{for}
$i=1,\ldots,s$'') running over the different r.h.s.\ and the inner
loop (``\textbf{for} $k=1,\ldots$ \textbf{until} convergence'')
performing the iteration of the Krylov subspace method. These two
loops can be interchanged, resulting in an algorithm where each
sweep through the inner loop can be organized in a way such that all
occurring matrix-vector multiplications for all $i$ are gathered into
one matrix-block-vector multiplication, the term ``block vector''
denoting a matrix of size $n \times s$. For \bicg as the underlying
single r.h.s.\ Krylov subspace method, Algorithm~\ref{alg:li_BiCG}
states the resulting \libicg method, where Li stands for ``loop
interchanged''. Note that whenever possible we cast inner loop
operations as operations with $n \times s$ matrices with, e.g.,
column $x_i^{(k)}$ of $X^{(k)} \in \mathC^{n \times s}$ being the
\bicg iterate for the r.h.s.\ $b_i$, etc. The inner loop index appears
explicitly only when calculating the $i$-th diagonal entries
$\alpha_i^{(k)}$ and $\beta_i^{(k)}$ of the $s \times s$ diagonal
matrices $\diag(\alpha^{(k)})$ and $\diag(\beta^{(k)})$,
respectively. All multiplications with the matrices $A$ and $A^H$
appear as matrix-block-vector multiplications $AP^{(k)}$ and
$A^H\hat{P}^{(k)}$. If the nominator in the definition of
$\alpha_i^{(k)}$ or $\beta_i^{(k)}$ is zero, the iteration for
r.h.s.\ $b_i$ breaks down.

\begin{algorithm}
\NoSemicolon
choose $X^{(0)}, \hat{R}^{(0)} \in \mathC^{n \times s}$ \\
put $R^{(0)} = B-AX^{(0)}, \, P^{(0)} = R^{(0)}, \hat{P}^{(0)} = \hat{R}^{(0)}$ \\
\For{$k=0,1,2,\dots$ {\rm until convergence}} {
 $Q^{(k)} = A P^{(k)}$ \\
 $\alpha_i^{(k)} = \langle r_i^{(k)}, \hat{r}_i^{(k)} \rangle / \langle q_i^{(k)}, \hat{p}_i^{(k)} \rangle, \; i=1,\ldots,s$ \\
 $X^{(k+1)} = X^{(k)} + P^{(k)}\diag(\alpha^{(k)})$ \\
 $R^{(k+1)} = R^{(k)} - Q^{(k)}\diag(\alpha^{(k)}), \; \hat{R}^{(k+1)} = \hat{R}^{(k)} - (A^H\hat{P}^{(k)})\overline{\diag(\alpha^{(k)})}$ \\
 $\beta_i^{(k)} = \langle r_i^{(k+1)}, \hat{r}_i^{(k+1)} \rangle / \langle r_i^{(k)}, \hat{r}_i^{(k)} \rangle , \; i=1,\ldots,s$ \\
 $P^{(k+1)} = R^{(k+1)} + P^{(k)}\diag(\beta^{(k)}), \; \hat{P}^{(k+1)} = \hat{R}^{(k+1)} + \hat{P}^{(k)}\overline{\diag(\beta^{(k)})}$
 }
\caption{Loop-interchanged \bicg (\libicg) \label{alg:li_BiCG}}
\end{algorithm}

{\em Global methods} take a different approach. They cast the $s$
systems \eqref{eq:simultaneous_systems} into one big, tensorized $sn
\times sn$ system
\begin{equation} \label{eq:global}
(I \otimes A) x = b, \mbox{ where } b = \left[ \begin{array}{c} b_1
\\ \vdots \\ b_{s} \end{array} \right], \, x = \left[ \begin{array}{c}
x_1 \\ \vdots \\ x_{s} \end{array} \right],
\end{equation}
and then apply a standard Krylov subspace method to the system
\eqref{eq:global}. Using the $\mat_{m\times s}$ operator to cast
vectors $x$ of length $ns$ into matrices $X$ of size $n \times s$
``by columns'', i.e.\ $(X)_{ij} = x_{(j-1)n+i}$, and the identities
\begin{align*}
  &\mat_{n \times s}((I \otimes A)x) = AX,  \enspace \langle x, y \rangle = \trace(Y^HX)\enspace \mbox{ where}\\
  &X = \mat_{n \times s}(x), Y = \mat_{n \times s}(y),
\end{align*}
we end up with the formulation of \bicg for \eqref{eq:global} given
in Algorithm~\ref{alg:gl_BiCG} known as global \bicg (\glbicg), see
\cite{Jbilou05}. The algorithm breaks down if the nominator in
$\alpha^{(k)}$ or $\beta^{(k)}$ is zero.

\begin{algorithm}
\NoSemicolon 
choose $X^{(0)}, \hat{R}^{(0)} \in \mathC^{n \times s}$ \\
put $R^{(0)} = B-AX^{(0)}, \, P^{(0)} = R^{(0)}, \hat{P}^{(0)} = \hat{R}^{(0)}$ \\
\For{$k=0,1,2,\dots$ {\rm until convergence}} {
 $Q^{(k)} = A P^{(k)}$ \\
 $\alpha^{(k)} = \trace((\hat{R}^{(k)})^HR^{(k)}) / \trace((\hat{P}^{(k)})^HQ^{(k)})$ \\
 $X^{(k+1)} = X^{(k)} + \alpha^{(k)} P^{(k)}$ \\
 $R^{(k+1)} = R^{(k)} - \alpha^{(k)} Q^{(k)}, \; \hat{R}^{(k+1)} = \hat{R}^{(k)} - \overline{\alpha^{(k)}}(A^H\hat{P}^{(k)})$ \\
 $\beta^{(k)} = \trace((\hat{R}^{(k+1)})^HR^{(k+1)})/ \trace((\hat{R}^{(k)})^HR^{(k)})$ \\
 $P^{(k+1)} = R^{(k+1)} + \beta^{(k)}P^{(k)}, \; \hat{P}^{(k+1)} = \hat{R}^{(k+1)} + \overline{\beta^{(k)}}\hat{P}^{(k)}$
 }
\caption{Global \bicg (\glbicg) \label{alg:gl_BiCG}}
\end{algorithm}

In the loop interchange and the global variant of \bicg the $i$th
columns $r^{(\ell)}_i$ for $\ell = 0,\ldots,k-1$ represent a basis
for the Krylov subspace
\[
\mathcal{K}_k(A,r_i^{(0)}) = \mbox{span} \{r_i^{(0)},
Ar_i^{(0)},\ldots A^{(k-1)}r_i^{(0)}\},
\]
and in both methods the iterate $x_i^{(k)}$, the $i$th column of
$X^{(k)}$, is taken from $x^{(0)}_i + \mathcal{K}_k(A,r_i^{(0)})$.
{\em Block} Krylov subspace methods take advantage of the fact that
the Krylov subspaces $\mathcal{K}_k(A,r_i^{(0)}), i=1,\ldots,s$ are
available simultaneously and take each iterate $x_i^{(k)}$ from the
larger ``block'' Krylov subspace
\[
\mathcal{K}_k(A,R^{(0)}) := \sum_{i=1}^s \mathcal{K}_k(A,r_i^{(0)}),
\]
which has dimension $sk$ (or less if linear dependence occurs ``across''
the single r.h.s.\
Krylov subspaces $\mathcal{K}_k(A,r_i^{(0)})$).
The usual approach is to transfer the variational
characterization for the iterates of the single r.h.s.\ method to the
block Krylov subspace. In this manner, the block \bicg method from
\cite{Oleary80} defines its iterates $X^{(k)}$ (they appear again as
the columns of an $n \times s$ block vector) by requiring that the
iterates $x_i^{(k)}$ are from $x_i^{(0)}+ \mathcal{K}_k(A,R^{(0)})$
and their residuals are orthogonal to
$\mathcal{K}_k(A^H,\hat{R}^{(0)})$ for some ``shadow residual''
block vector $\hat{R}^{(0)} \in \mathC^{n \times s}$. The resulting
method \blbicg (block \bicg), is given as
Algorithm~\ref{alg:bl_BiCG}. Note that now the quantities
$\alpha^{(k)}, \hat{\alpha}^{(k)}, \beta^{(k)}$ and
$\hat{\beta}^{(k)}$ are $s \times s$-matrices, and the method breaks
down prematurely if one of the matrices $(P^{(k)})^HA^H
\hat{P}^{(k)}, (\hat{R}^{(k)})^HR^{(k)}$ is singular.

\begin{algorithm}
\NoSemicolon
choose $X^{(0)}, \hat{R}^{(0)} \in \mathC^{n \times s}$ \\
put $R^{(0)} = B-AX^{(0)}, \, P^{(0)} = R^{(0)}, \hat{P}^{(0)} = \hat{R}^{(0)}$ \\
\For{$k=0,1,2,\dots$ {\rm until convergence}} {
 $Q^{(k)} = A P^{(k)}, \; \hat{Q}^{(k)} = A^H\hat{P}^{(k)}$ \\
 $\alpha^{(k)} = \big((\hat{P}^{(k)})^HQ^{(k)}\big)^{-1}\big((\hat{R}^{(k)})^HR^{(k)}\big)$ \\
 $\hat{\alpha}^{(k)} =
 \big((P^{(k)})^H \hat{Q}^{(k)}\big)^{-1}\big(({R}^{(k)})^H\hat{R}^{(k)}\big)$ \\
 $X^{(k+1)} = X^{(k)} +  P^{(k)}\alpha^{(k)}$ \\
 $R^{(k+1)} = R^{(k)} -  Q^{(k)}\alpha^{(k)}, \; \hat{R}^{(k+1)} = \hat{R}^{(k)} - \hat{Q}^{(k)}\hat{\alpha}^{(k)}$ \\
 $\beta^{(k)} = \big((\hat{R}^{(k)})^HR^{(k)}\big)^{-1}\big((\hat{R}^{(k+1)})^HR^{(k+1)}\big)$ \\
 $\hat{\beta}^{(k)} = \big(({R}^{(k)})^H\hat{R}^{(k)}\big)^{-1}\big(({R}^{(k+1)})^H\hat{R}^{(k+1)}\big) $\\
 $P^{(k+1)} = R^{(k+1)} + P^{(k)}\beta^{(k)}, \; \hat{P}^{(k+1)} = \hat{R}^{(k+1)} + \hat{P}^{(k)}\hat{\beta}^{(k)}$
 }
\caption{Block \bicg (\blbicg) \label{alg:bl_BiCG}}
\end{algorithm}

Apart from the matrix-vector multiplications, the work to update the
various quantities is substantially higher in block \bicg as compared
to loop interchanged \bicg and global \bicg. Counting an inner product
or a SAXPY operation with vectors of length $n$ as one {\em vector
operation} (representing $n$ additions and $n$ multiplications), we
have the following proposition, the proof of which also gives an
indication on how to save arithmetic work by re-using information
from the previous iteration.

\begin{proposition} \label{prop:cost}
One iteration of \libicg, \glbicg or \blbicg requires two
matrix-block-vector multiplications (one with $A$ and one with
$A^H$) with block-vectors of size $n \times s$ plus $7s$ additional
vector operations for \libicg as well as \glbicg, and $7s^2$
additional vector operations for \blbicg.
\end{proposition}
\begin{proof}
We only have to care about operations other than the multiplications
with $A$ and $A^H$. In \libicg, assuming that we save $\langle
\hat{r}_i^{(k+1)}, r_i^{(k+1)}\rangle$ for re-use in the next
iteration, we need a total of $2s$ inner products to compute all
$\alpha_i^{(k)}, \beta_i^{(k)}$ and $s$ SAXPYs for each of
$X^{(k+1)}$, $R^{(k+1)}$, $\hat{R}^{(k+1)}$, $P^{(k+1)}$,
$\hat{P}^{(k+1)}$. Exactly the same count holds for \glbicg. In
\blbicg we can save $\big(\hat{R}^{(k+1)}\big)^HR^{(k+1)}$ for use
in the next iteration, and we can exploit the fact that the two
factors defining $\alpha^{(k)}$ are just the adjoints of those
defining $\hat{\alpha}^{(k)}$, and similarly for $\beta^{(k)}$ and
$\hat{\beta}^{(k)}$.  So we need just $2s^2$ inner products to build
these factors, and we neglect the cost $\mathcal{O}(s^3)$ for
multiplying $s\times s$ matrices. The computation of each of
$X^{(k+1)}$, $R^{(k+1)}$, $\hat{R}^{(k+1)}$, $P^{(k+1)}$,
$\hat{P}^{(k+1)}$ requires $s^2$ SAXPYs.
\end{proof}

We note that the updates of $X^{(k+1)}$, $R^{(k+1)}$,
$\hat{R}^{(k+1)}$, $P^{(k+1)}$, $\hat{P}^{(k+1)}$ in the block
method actually represent BLAS3 GEMM operations, see \cite{Blas3}
which have a more favorable ratio of memory access to arithmetic
work than SAXPY operations, so the overhead of a factor of $s$ of
the block method vs.\ the loop interchange and the global method
suggested by Proposition~\ref{prop:cost} may show less in actual
timings.

In a similar way one obtains the three simultaneous variants \liqmr,
\glqmr and \blqmr of the \qmr method. The variational
characterization in \blqmr is that the 2-norm of the coefficients
which represent the residual from $K_k(A,R^{(0)})$ in the nested
bi-orthogonal basis of $K_k(A,R^{(0)})$ with respect to
$K_k(A^H,\hat{R}^{(0)})$, be minimal; see \cite{Freund97}. We do not write down
the resulting algorithms explicitly nor do we so for the \bicgstab
variants \libicgstab, \glbicgstab and \blbicgstab. It is worth
mentioning, though, that \bicgstab does not have a proper variational
characterization, its main property being that multiplications with
$A^H$ present in \bicg are replaced by those with $A$, thus obtaining
iterate $x^{(k)}$ from $x^{(0)} + \mathcal{K}_{2k}(A,r^{(0)})$ while
imposing a steepest descent property on the residual as an
intermediate step. The block \bicgstab method from \cite{ELg03}
transfers this approach to the multiple r.h.s.\ case such that the
$k$-th iterate $x_i^{(k)}$ is from $x_i^{(0)} +
\mathcal{K}_{2k}(A,R^{(0)})$, imposing a condition on the residuals
in the intermediate steps which, interestingly, is quite in the
spirit of a ``global'' steepest descent method. We refer to
\cite{ELg03}
for details.

\section{Discussion of the literature} \label{sec:literature}
The loop interchange approach is most probably not new, although we are
not aware of a systematic discussion as a construction principle for
simultaneous methods.

Global methods were considered in a variety of papers. Surprisingly,
although they all just realize the approach to perform the
respective single r.h.s.\ method for \eqref{eq:global} and then
``matricize'' all operations with vectors of length $ns$, we could
not find this fact mentioned explicitly in the literature.
The first global methods are global full orthogonalization
(Gl-FOM)  and global generalized minimal residual (Gl-GMRES),
introduced in \cite{Jbilou99}. \glbicg and \glbicgstab were
suggested in \cite{Jbilou05}, and global variants of less well-known
Krylov subspace methods were subsequently proposed in \cite{Hey01}
(Gl-CMRH), \cite{Zhang10} (Gl-CGS), \cite{Yang07} (Gl-SCD) and
\cite{Zhang11} (Gl-BiCR and its variants).

Block Krylov subspace methods were introduced in \cite{Oleary80},
where \blbicg was considered together with block conjugate gradients
(\blcg). The (long recurrence) block generalized minimal residual
(\blgmres) algorithm goes back to \cite{SimonciniGallopoulos1995} and
\cite{Vital1990}, and a convergence analysis can be
found in \cite{SimonciniGallopoulos1996}. Block Krylov subspace
methods require additional work for factorizations and multiplications of $n\times s$ or
$s \times s$ matrices in each iteration
(cf.~Proposition~\ref{prop:cost}).
Very often, this substantially slows down the overall process.
The hybrid method from \cite{Simo96} describes an approach,
where this additional work is reduced by adding cycles in which a (matrix
valued)
polynomial obtained from the block Arnoldi process of a previous
cycle is used. The additional cycles save arithmetic cost since they
do not perform the block Arnoldi process.

The \blqmr method was suggested in \cite{Freund97}, \blbicgstab goes
back to \cite{ELg03}, and Bl-LSQR, a block least squares QR-algorithm to solve the normal
equations for simultaneous r.h.s.\ was given in \cite{Karimi06}. The
block methods relying on the non-symmetric Lanczos process can
suffer from the fact that this process can break down prematurely if
the bi-orthogonality condition imposed on the to be computed block
basis vectors can not be fulfilled. Numerically, this will typically
result in very ill-conditioned $s \times s$ matrices in \blbicg and
\blbicgstab for which linear systems have to be solved in the
algorithm. The non-symmetric Lanczos process at the basis of \blqmr
from \cite{deflation1} does, in principle, not suffer from this phenomenon since it
incorporates a look-ahead strategy which cures almost all possible
unwanted breakdowns.\footnote{This cure is somewhat cumbersome to implement, though.}
Both, the block Arnoldi process at the basis of
\blgmres as well as the symmetric or non-symmetric Lanczos process
at the basis of \blcg or \blbicg, \blqmr and \blbicgstab,
respectively, should account for an additional danger of producing
ill-conditioned $s \times s$ systems arising because of deflation.
Occurrence of deflation within a block Krylov subspace
method means that while the block Krylov subspaces
$\mathcal{K}_k(A,R^{(0)})$ or $\mathcal{K}_k(A^H,\hat{R}^{(0)})$ has dimension
$ks$, the next,
$\mathcal{K}_{k+1}(A,R^{(0)})$ or $\mathcal{K}_{k+1}(A^H,\hat{R}^{(0)})$,
has dimension less than $(k+1)s$. Similar reductions in
dimension might occur again in later iterations. In principle,
deflation is advantageous, since it allows to reduce the number of
matrix-vector operations per iteration. However, algorithms have to
be adjusted, and additional book-keeping is necessary. Deflation in
the unsymmetric block Lanczos process and the block Arnoldi process was
considered in \cite{deflation2}, the consequences for
\blgmres are discussed in detail in \cite{Gut07}. \blqmr from
\cite{Freund97} and the \blcg variant considered in \cite{Birk14}
address deflation by indeed explicitly reducing the dimension of the
block Krylov subspaces. Interestingly, the latter two methods now work by
applying the matrix-vector multiplications one at a time instead
of simultaneously to a block vector, checking for deflation after
each such operation. These methods can thus no longer take advantage
of a possible performance gain due to multiplications of matrices
with block-vectors. This is in contrast to the modifications of \blbicg and
\blcg from \cite{Oleary80} and \cite{Dub01} which ``hide'' the possible
singularity of
the $s \times s$ matrices by using QR-decompositions and modified
update formulae. This approach to treat deflation implicitly does
not allow to save matrix-vector multiplications when deflation
occurs, but keeps the advantage of relying on matrix-block-vector
multiplications. It does not require any additional book-keeping.
Recently, in \cite{Birk2015} a variant of \blcg was developed which
treats deflation explicitly, but still uses matrix-block-vector
multiplications.

A round-off error analysis for \blbicgstab in \cite{Tad09} lead the
authors to suggest a numerically more stable modification which
basically interchanges the minimization step and the \bicg step present in each
iteration of \glbicgstab.
A numerically more stable variant of \blbicg which enforces
$A$-orthogonality between the individual vectors and not just between block
vectors, and which contains an additional \qmr-type stabilization step, was
considered in \cite{Simoncini97}.

\section{Improvements: Cost and stability} \label{sec:improvements}
In any of the three simultaneous \bicg methods we are free to choose
$\hat{R}^{(0)}$, the initial block-vector of shadow residuals,
popular choices being $\hat{R}^{(0)} = R^{(0)}$ or $\hat{R}^{(0)} =
\overline{R^{(0)}}$. If we take $\hat{R}^{(0)}$ to have identical
columns, i.e.\
\[
\hat{R}^{(0)} = \hat{r}^{(0)} \one^T, \mbox{ where } \hat{r}^{(0)}
\in \mathC^n, \; \one = (1,\ldots,1)^T \in \mathC^s,
\]
we see that in \glbicg (Algorithm~\ref{alg:gl_BiCG}) all ``shadow''
block vectors $\hat{P}^{(k)}$ and $\hat{R}^{(k)}$ conserve this
structure for all $k$, i.e.\
\[
\hat{P}^{(k)} = \hat{p}^{(k)} \one^T, \; \hat{R}^{(k)} =
\hat{r}^{(k)}\one^T, \enspace \hat{p}^{(k)}, \hat{r}^{(k)} \in \mathC^n.
\]
A comparable situation occurs neither in \libicg nor in
\blbicg. Only in \glbicg can we therefore take advantage of an
initial shadow residual with identical columns and just compute the
vectors $\hat{p}^{(k)}$ and $\hat{r}^{(k)}$ rather than the
respective block vectors in iteration $k-1$. In particular, we then
need to multiply $A^H$ only with a single vector instead of a whole
block vector with $s$ columns. The resulting ``economic'' version
\eglbicg of \glbicg is given as Algorithm~\ref{alg:egl_BiCG}, where
we used the $\sumv$-operator to denote the sum of all components of
a (row) vector and the relation
  $\trace(\one \hat{r}^H R) = \sumv(\hat{r}^HR)$
for
  $\one \in \mathC^s, \hat{r} \in \mathC^n, R \in \mathC^{n \times s}$.

\begin{algorithm}
\NoSemicolon
choose $X^{(0)} \in \mathC^{n \times s}, \;  \hat{r}^{(0)} \in \mathC^n$ \\
put $R^{(0)} = B-AX^{(0)}, \, P^{(0)} = R^{(0)}, \hat{p}^{(0)} = \hat{r}^{(0)}$ \\
\For{$k=0,1,2,\dots$ {\rm until convergence}} {
 $Q^{(k)} = A P^{(k)}$ \\
 $\alpha^{(k)} = \sumv((\hat{r}^{(k)})^HR^{(k)}) / \sumv((\hat{p}^{(k)})^HQ^{(k)})$ \\
 $X^{(k+1)} = X^{(k)} + \alpha^{(k)} P^{(k)}$ \\
 $R^{(k+1)} = R^{(k)} - \alpha^{(k)} Q^{(k)}, \; \hat{r}^{(k+1)} = \hat{r}^{(k)} - \overline{\alpha^{(k)}}(A^H\hat{p}^{(k)})$ \\
 $\beta^{(k)} = \sumv((\hat{r}^{(k+1)})^HR^{(k+1)})/ \sumv((\hat{r}^{(k)})^HR^{(k)})$ \\
 $P^{(k+1)} = R^{(k+1)} + \beta^{(k)}P^{(k)}, \; \hat{p}^{(k+1)} = \hat{r}^{(k+1)} + \overline{\beta^{(k)}}\hat{p}^{(k)}$
 }
\caption{Economic global \bicg (\eglbicg) \label{alg:egl_BiCG}}
\end{algorithm}

In a similar way, we obtain an economic version, \eglqmr, of the
global \qmr method. There is no such opportunity for \glbicgstab,
where multiplications with $A^H$ are not present anyway. In the
economic global methods the work to be performed for the shadow
residuals and search directions is identical to that for just one
single r.h.s., so that it becomes increasingly negligible for
larger values of $s$, $s \gtrapprox 10$ say. In this manner, the
economic global methods eliminate one of the important disadvantages
of the short recurrence Krylov subspace methods based on the
non-symmetric Lanczos process, and we will see in the numerical
experiments that \eglbicg and \eglqmr perform favorably when
compared with simultaneous \bicgstab variants, for example.

Deflation in \blbicg occurs if at some iteration $k$
one of the block vectors $P^{(k)}, \hat{P}^{(k)}, R^{(k)},$ and $\hat{R}^{(k)}$
becomes rank deficient, even though one might have
taken care for this not to happen in iteration 0.
In practice, exact deflation happens rarely, but if one of the block vectors
is almost rank deficient (inexact deflation),
some of the matrices $\alpha^{(k)}, \hat\alpha^{(k)}, \beta^{(k)},$ and $\hat\beta^{(k)}$
will be computed inaccurately,
resulting in numerical instabilities of the algorithm. Already in
\cite{Oleary80} it was therefore proposed to use a QR factorization of the
search direction block vectors $P^{(k)}$ and $\hat{P}^{(k)}$ to avoid such
instabilities. A systematic study of different other remedies was
presented in \cite{Dub01} for the Hermitian positive definite case, i.e., for
the block CG method, and
it turned out that the variant which uses a QR factorization of the residual
block vectors is most convenient. This approach can be transported to block
\bicg, and we performed a series of numerical comparisons which shows that
also in the non-Hermitian case the variant which relies on a QR
factorization of the block residuals is to be preferred over the one with
QR-factorization of the block search vectors.

To precisely describe the resulting variant of block \bicg, consider
a (thin) QR-factorization of the block residuals in \blbicg
(Algorithm~\ref{alg:bl_BiCG}),
\begin{equation} \label{eq:QR}
  R^{(k)} = {Q}^{(k)} {C}^{(k)},\quad \hat R^{(k)} =  {\hat Q}^{(k)}{\hat C}^{(k)},
\end{equation}
where
\[
  Q^{(k)}, \hat{Q}^{(k)} \in \mathC^{n \times s}, \enspace
  (Q^{(k)})^HQ^{(k)} = {(\hat{Q}^{(k)})}^H \hat{Q}^{(k)} = I, \enspace
  C^{(k)},  \hat{C}^{(k)} \in \mathC^{s \times s}.
\]
A possible ill-conditioning of $R^{(k)}$ or $\hat{R}^{(k)}$ translates into
an ill-conditioned matrix $C^{(k)}$ or $\hat{C}^{(k)}$, respectively, and
\blbicg can now be stabilized by using the QR-factorizations \eqref{eq:QR} while at the
same time avoiding the occurrence of $(C^{(k)})^{-1}$ and
$(\hat{C}^{(k)})^{-1}$. To do so, we represent the search direction block vectors
$P^{(k)}$ and $\hat{P}^{(k)}$ as
\[
P^{(k)}=V^{(k)}C^{(k)} ,\quad \hat P^{(k)}=\hat V^{(k)}\hat C^{(k)}.
\]
For the update of the block vectors $V^{(k)}$ and
$\hat{V}^{(k)}$ from Algorithm~\ref{alg:bl_BiCG} we then get
\begin{eqnarray*}
V^{(k+1)} &=& Q^{(k+1)}+V^{(k)} \big( C^{(k)} \beta^{(k)} (C^{(k+1)})^{-1}
\big),\\
\hat V^{(k+1)} &=& \hat Q^{(k+1)}+\hat V^{(k)} \big( \hat{C}^{(k)}
\hat{\beta}^{(k)} (\hat{C}^{(k+1)})^{-1} \big),
\end{eqnarray*}
and using \eqref{eq:QR} in the update for the block residuals we obtain
\begin{eqnarray*}
Q^{(k+1)}C^{(k+1)}{(C^{(k)})}^{-1} &=&
Q^{(k)}-AV^{(k)}\big(C^{(k)}\alpha^{(k)}(C^{(k)})^{-1}\big), \\
\hat{Q}^{(k+1)}\hat{C}^{(k+1)}{(\hat{C}^{(k)})}^{-1} &=&
\hat{Q}^{(k)}-A^H\hat{V}^{(k)}\big(\hat{C}^{(k)}\hat{\alpha}^{(k)}
(\hat{C}^{(k)})^{-1}\big).
\end{eqnarray*}
This shows that computationally we can obtain
$S^{(k+1)}=C^{(k+1)}{(C^{(k)})}^{-1}$ together with $Q^{(k+1)}$ from a
QR factorization
of $Q^{(k)}-AV^{(k)}\big(C^{(k)}\alpha^{(k)}(C^{(k)})^{-1}\big)$, and
similarly
for the ``shadow'' block vectors. Moreover, we have
\begin{eqnarray}
C^{(k)}\alpha^{(k)}(C^{(k)})^{-1} &=& C^{(k)}
\big((\hat{P}^{(k)})^HAP^{(k)}\big)^{-1}\big((\hat{R}^{(k)})^HR^{(k}\big)
{(C^{(k)})}^{-1} \nonumber \\
&=& \big((\hat{V}^{(k)})^HAV^{(k)}\big)^{-1}\big((\hat{Q}^{(k)}
)^HQ^{(k}\big), \label{eq:newalpha}
\end{eqnarray}
and, analogously,
\begin{eqnarray}
\hat{C}^{(k)}\hat\alpha^{(k)}(\hat{C}^{(k)})^{-1}
& = & \big(({V}^{(k)})^HA^H\hat
V^{(k)}\big)^{-1}\big(({Q}^{(k)})^H\hat Q^{(k)}\big), \label{eq:newalphahat}\\
C^{(k)}\beta^{(k)}(C^{(k+1)})^{-1}
& = & \big((\hat{Q}^{(k)})^HQ^{(k)}\big)^{-1} (\hat
S^{(k+1)})^H\big((\hat{Q}^{(k+1)})^HQ^{(k+1)}\big), \label{eq:newbeta}\\
\hat C^{(k)}\hat \beta^{(k)}(\hat C^{(k+1)})^{-1} & = &
\big(({Q}^{(k)})^H\hat Q^{(k)}\big)^{-1}
(S^{(k+1)})^H\big(({Q}^{(k+1)})^H\hat Q^{(k+1)}\big)
\label{eq:newbetahat}.
\end{eqnarray}
Putting things together, we arrive at the block \bicg algorithm using
QR-factorization of the block residual, termed \blbicgrq, given as
Algorithm~\ref{alg:bl_BiCG_rQ}.

\begin{algorithm}
\NoSemicolon
choose $X^{(0)}, \hat{R}^{(0)} \in \mathC^{n \times s}$ \\
put $R^{(0)} = B-AX^{(0)}, \, Q^{(0)}C^{(0)} = R^{(0)}, \, \hat Q^{(0)}\hat C^{(0)} = \hat R^{(0)},\,V^{(0)} = Q^{(0)}, \hat{V}^{(0)} = \hat{Q}^{(0)}$ \\
\For{$k=1,2,\dots$ {\rm until convergence}} {
 $W^{(k)} = A V^{(k)}, \; \hat{W}^{(k)} = A^H\hat{V}^{(k)}$ \\
 $\alpha^{(k)} =\big((\hat{V}^{(k)})^H W^{(k)}\big)^{-1}\big((\hat{Q}^{(k)})^H Q^{(k)}\big)$ \\
$\hat\alpha^{(k)}=\big(({V}^{(k)})^H \hat W^{(k)}\big)^{-1}\big(({Q}^{(k)})^H\hat Q^{(k)}\big)$ \\
$X^{(k+1)} = X^{(k)} +  V^{(k)}\alpha^{(k)}C^{(k)}$ \\
$Q^{(k+1)}S^{(k+1)}=Q^{(k)}-W^{(k)}\alpha^{(k)}, \; C^{(k+1)}=S^{(k+1)}C^{(k)}$\\
$\hat Q^{(k+1)}\hat S^{(k+1)}=\hat Q^{(k)}-\hat W^{(k)}\hat \alpha^{(k)},\; \hat C^{(k+1)}=\hat S^{(k+1)}\hat C^{(k)}$\\
$\beta^{(k)} = \big((\hat{Q}^{(k)})^HQ^{(k)}\big)^{-1} (\hat S^{(k+1)})^H\big((\hat{Q}^{(k+1)})^HQ^{(k+1)}\big)$ \\
 $\hat{\beta}^{(k)} =\big(({Q}^{(k)})^H\hat Q^{(k)}\big)^{-1}
(S^{(k+1)})^H\big(({Q}^{(k+1)})^H\hat Q^{(k+1)}\big) $\\
 $V^{(k+1)} = Q^{(k+1)} + V^{(k)}\beta^{(k)}, \; \hat{V}^{(k+1)} = \hat{Q}^{(k+1)} + \hat{V}^{(k)}\hat{\beta}^{(k)}$
 }
\caption{Block \bicg with QR factorization of residual block vectors (\blbicgrq)
\label{alg:bl_BiCG_rQ}}
\end{algorithm}

Several remarks are in order: First, in
Algorithm~\ref{alg:bl_BiCG_rQ} we re-used the symbols
$\alpha^{(k)}, \beta^{(k)}$ etc.\ which now designate the quantities
at the right of \eqref{eq:newalpha} - \eqref{eq:newbetahat}. Second,
although the block residuals $R^{(k)}$ are no more available
explicitly in the algorithms, their 2-norms and Frobenius norms can
still be easily retrieved and used in a stopping criterion, since
both these norms are invariant under orthogonality transformations
and thus
\[
\|R^{(k)}\| = \| Q^{(k)} C^{(k)} \| = \| C^{(k)} \|,
\]
$C^{(k)}$ being available from the algorithm. Third,
Algorithm~\ref{alg:bl_BiCG_rQ}
requires more work than Algorithm~\ref{alg:bl_BiCG}, mainly because of
the additional two QR-factorizations of block vectors required in each
iteration. They have at least a cost of 
$3ns^2 +
\mathcal{O}(ns)$ each (using, e.g., the modified Gram-Schmidt algorithm).
Finally, all block vectors
in Algorithm~\ref{alg:bl_BiCG_rQ} now always have full rank, thus reducing one
source of possible ill-conditioning in $\alpha^{(k)}, \beta^{(k)}$, etc.
In the case of $A$ being Hermitian and positive definite, where
\bicg can be reduced to CG, this is an exhaustive cure to possible
ill-conditioning, see
\cite{Dub01}. In \bicg, however, it can still happen that, e.g.,
$(\hat{Q}^{(k)})^HQ^{(k)}$ or $(\hat{V}^{(k)})^HW^{(k)}$ is ill-conditioned or
singular. This is inherent to the bi-orthogonality condition on which the whole
method is built, and can be avoided only if one is willing to deviate from the
bi-orthogonality condition by, for example, modifying the method using a
look-ahead version of the unsymmetric block Lanczos process.

Relying on the deflation procedure from \cite{deflation1} within a look-ahead
block Lanczos process, the block \qmr method from \cite{Freund97}
addresses possible rank-deficiency in the block residuals as well as
possible further breakdowns related to the bi-orthogonality condition.
In this approach, the basis vectors for the next Krylov subspace are
built one at a time and are then checked for deflation, so that this
approach does not allow to compute matrix-block-vector products.

To conclude this section, we remark that it is possible to, in a
similar spirit, use a QR-factorization of the residuals in the block
\bicgstab method. We do not write down the resulting algorithm,
\blbicgstabrq explicitly, but we will report on its performance in
our numerical experiments. A corresponding variant based on QR
factorization of the search directions performed less well in our
experiments, as was the case in \blbicg.

\section{Numerical experiments} \label{sec:experiments}

The purpose of this section is to assess the efficiency and stability
of the various simultaneous methods. To this end we performed numerical
experiments for five different examples in which systems with several r.h.s.\
arise naturally. These examples will be described in detail below. All iterations were started with $X^{(0)} = 0$.
In all examples we preprocessed the block vector $B$ of r.h.s.\ by computing its QR-factorization, $B=QR$ and then replaced $B$ by $Q$.
The initial shadow block residual was taken equal to the initial residual, $\hat{R}^{(0)} = R^{(0)} \; (= Q)$, except for the economic
versions where we took $\hat{r}^{(0)}$ as the arithmetic mean of all initial residuals.
The stopping criterion was always $\|R^{(k)}\|_F \leq 10^{-10} \cdot \|R^{(0)}\|_F$.

All experiments were run in Matlab on an Intel Core i7-4770 quad core processor.
Most of our Matlab code is based on pre-compiled Matlab functions and routines, e.g.\ for computing sparse matrix-vector products,
factorizations or products of matrices. We therefore trust that the reported timings
are significant with the possible exception of the loop interchanged methods where the explicitly occurring inner loop
over the r.h.s.\ introduces a non-negligible part of interpreted code, thus increasing the compute time.

For our Matlab experiments, the ratio $a$ of the time required to perform $s$ single matrix-vector multiplications and that for a matrix-block vector multiplication with block size $s$ ranges between $1.2$ and $2.7$.
The benefit of working with block vectors is thus substantial.
In our case, this gain is primarily due to the fact that Matlab uses all cores for block vectors, whereas it uses only one core for a single vector.
On other processors and with other compilers similar gains do occur albeit sometimes for different reasons. We refer to \cite{welleinetal2015} for a 
recent publication where the benefits of matrix-block vector multiplications in the presence of limited memory bandwidth are investigated based on 
an appropriate ``roofline'' performance model.   

The additional work to be performed in the block methods
(see Proposition~\ref{prop:cost}) is of the order of $\mathcal{O}(s^2)$ vector operations,
$s$ the number of right hand sides. This must be considered relative to the work
required for $s$ matrix-vector products, which amounts to $s \cdot \mathrm{nnz}/n$ vector operations,
$\mathrm{nnz}$ being the number of non-zeros in $A$. Therefore, the ratio
\begin{equation} \label{eq:rho}
\rho = \frac{s^2}{s \cdot \mathrm{nnz} /n} = s \cdot \frac{n}{\mathrm{nnz}}
\end{equation}
can be regarded as a quantity measuring the (relative) additional cost for $s$ r.h.s.\ in the
block methods. If $\rho$ is large, the additional arithmetic cost is substantial.

Notwithstanding the detailed description of our examples below, Table~\ref{tab:rhoetc} reports the
values of the various quantities just discussed for these examples. For Examples~\ref{ex:ex1}-\ref{ex:ex4} we
will also report results using ILU preconditioning.
Besides from accelerating the convergence, preconditioning increases the cost of a matrix-(block)-vector multiplication
which is now a preconditioned matrix-block vector multiplication. This affects the ratio $a$ and reduces the ratio $\rho$.
Table~\ref{tab:rhoetc} therefore gives values for the unpreconditioned and the preconditioned cases.
Since none of the standard preconditioners was effective for Example~\ref{ex:ex5},
we only report results for the unpreconditioned case there.

\begin{table}
\caption{\label{tab:rhoetc}Number of r.h.s.\ $s$, measured acceleration factor $a$ (ratio of the time for $s$ matrix-vector multiplications to the time for one matrix-block vector multiplication with block size $s$), and value of $\rho$ from \eqref{eq:rho} for the numerical examples.}
\begin{center}
\addtolength\tabcolsep{-2.0pt}
\begin{tabular}{|r|cc|cc|cc|cc|c|} \hline
        & \multicolumn{2}{|c|}{Ex.~\ref{ex:ex1}}   &  \multicolumn{2}{|c|}{Ex.~\ref{ex:ex2}} &  \multicolumn{2}{|c|}{Ex.~\ref{ex:ex3}} &  \multicolumn{2}{|c|}{Ex.~\ref{ex:ex4}} & Ex.~\ref{ex:ex5}   \\
        & no prec. & prec.          & no prec. & prec.           & no prec. & prec.           & no prec. & prec.           & no prec. \\ \hline
 $s$    & \multicolumn{2}{|c|}{$4$} & \multicolumn{2}{|c|}{$19$} & \multicolumn{2}{|c|}{$19$} & \multicolumn{2}{|c|}{$12$} & $20$     \\
 $a$    & $2.65$   & $1.21$         & $2.66$   & $1.69$          & $2.60$   & $1.71$          & $1.57$   & $1.18$          & $1.38$   \\
 $\rho$ & $0.8$    & $0.36$         & $2.76$   & $1.29$          & $2.76$   & $1.29$          & $0.24$   & $0.14$          & $5.00$   \\ \hline
\end{tabular}
\addtolength\tabcolsep{2.0pt}
\end{center}
\end{table}

\subsection{Description of examples} We considered the following five examples.

\begin{example} \label{ex:ex1}
This is a variation of an example which was given in \cite{Reichel92}
in the context of Krylov subspace type methods for the Sylvester equation. We
consider the elliptic pde
\[
 -u_{xx} - u_{yy} +2\alpha_1u_x +2\alpha_2u_y - 2\alpha_3u = 0,
\]
on the unit square with Dirichlet boundary conditions, where
$\alpha_1=\alpha_2=\alpha_3=5$. We discretize using finite differences on an equispaced 
grid with $200 \times 200$ interior grid points.
We build a total of four r.h.s. associated with the four corners of $\Omega$.
For a given corner, we construct the r.h.s.\ such that it represents the boundary condition
of $u$ being continuous and piecewise linear
on the boundary  with value 1 at that corner and value 0 at all the other corners. From the four solutions
to these four r.h.s.\ we can thus, by superposition,
retrieve the solutions for any continuous piecewise linear boundary values for $u$.

This is an example of an elliptic pde with mild convection, the value for $\rho$ from \eqref{eq:rho} is
relatively small, $\rho = 0.8$.
\end{example}

\begin{example} \label{ex:ex2}
We consider the three-dimensional advection
dominated elliptic pde on the unit cube
\begin{equation} \label{eq:ex1}
 u_{xx} + u_{yy} + u_{zz} + \nu \cdot u_x = f, \enspace (x,y,z) \in \Omega =
(0,1)^3,
\end{equation}
where $\nu >0$ is a parameter which controls the influence of the
convection term and $f$ is defined such that for zero Dirichlet boundary
conditions the exact solution is $u(x,y,z) =
\exp(xyz)\sin(\pi x)\sin(\pi y)\sin(\pi z)$, see
\cite{Sle93} as an example where \bicgstab encounters convergence
problems for larger values of $\nu$.
We discretize \eqref{eq:ex1} using finite differences
on an equispaced grid with $50^3$ interior grid points. This results in a
linear system of size  $n=125,\!000$ with seven non-zeros per row. We generate
one r.h.s.\ by imposing zero Dirichlet boundary conditions on all faces of
$\Omega$. In order to
be able to (by superposition) retrieve solutions for any Dirichlet boundary
conditions which depend piecewise affinely on $x,y$ and
$z$ on any of the six faces of $\Omega$, we generate a total of 18 additional
r.h.s.\ with $f=0$ in \eqref{eq:ex1}. To be specific,
for the face with $x=0$, e.g., we set the boundary
conditions equal to $0$ on all other faces and obtain three r.h.s.\ by setting
the boundary condition on that face equal to $y$, to $z$ and to the constant
$1$, respectively, and similarly for the five other faces.
Our choice for $\nu$ is $\nu = 1,\!000$, as in \cite{Sle93}, resulting in a
convection dominated equation and a highly
non-symmetric system matrix. The value for $\rho$ from \eqref{eq:rho} is significantly larger than
in the first example, $\rho \approx 2.7$.
\end{example}

\begin{example}\label{ex:ex3}%
Same as Example~\ref{ex:ex1}, but now with $\nu = 10$. This means that the matrix is ``almost'' symmetric.
\end{example}

\begin{example}\label{ex:ex4}%
This example arises from quantum chromodynamics (QCD),
the physical theory of the strong interaction between the quarks as
constituents of matter. The Wilson-Dirac matrix $D_W$ in lattice QCD represents
a discretization of the Dirac operator from QCD, see \cite{FroKaKrLeRo13}, e.g.
One has $D_W = I - \kappa D$ with $D$ representing a periodic nearest neighbor coupling on a four-dimensional
lattice, and there are 12 variables per lattice point, one for each possible
combination of four spin and three color indices. The couplings are defined via
a gluon background field which is drawn from a statistical distribution
and therefore varies irregularly from one lattice point to the next. A
typical ``propagator computation'' in lattice QCD amounts to solve
\[
D_W x_i = e_i, \enspace i=1,\ldots,12,
\]
where $e_1,\ldots,e_{12}$ are the first $s=12$ unit vectors (corresponding to
the twelve variables at one lattice point).
For our computation we took a Wilson-Dirac matrix
based on a $8^4$ lattice and thus has dimension $12\cdot 8^4 = 49,\!152$.
Our configuration corresponds to a typical configuration in lattice QCD with temperature parameter $\beta=5.6$ and
coupling parameter $\kappa=0.16066$  
corresponding to run $A_3$ in \cite{DebbioGiustiLuscherEtAl2007,DebbioGiustiLuscherEtAl2007a}.
Similar matrices can be found e.g.\ as \texttt{conf5\_4-8x8-20} and \texttt{conf6\_0-8x8-80} from the University of Florida Sparse Matrix Collection \cite{DavisHu2011}.
The entries of $D_W$ are complex numbers, the matrix is non-symmetric with its
spectrum being symmetric to the real axis and located in the right half plane.
There are 49 non-zeros per row in $D_W$. This is our example with the smallest value for $\rho$ from
\eqref{eq:rho}, $\rho \approx 0.25$.
\end{example}

\begin{example}\label{ex:ex6}\label{ex:ex5}%
The FEAST method to compute some eigenpairs \cite{Polizzi2009} for the
generalized eigenvalue problem $Ax = \lambda Bx$ evaluates the contour integrals
  $\oint_\Gamma (A - t B)^{-1} BY dt$
where $\Gamma$ is a circle that contains the eigenvalues of the eigenpairs that
are to be computed and $Y\in \mathC^{n\times m}$ consists of $m$ randomly chosen
columns $y_i$.
Using numerical quadrature for the integral, this means that one has to solve
several linear systems
  $(A - t_j B) x_i = y_i$
for a given choice of the quadrature nodes $t_j$.
For our example, we took $A$ as the matrix
stemming from a model for graphene belonging to the parameter $W=200$ in
\cite{GalgonKraemerEtAl2014} with $n = 40,\!000$, and $B=I$. There, the eigenpairs corresponding to
the eigenvalues of smallest modulus are sought, i.e. $\Gamma$ is
centred at $0$ with a radius of $0.125$. We solve the systems
  $(A - tI) x_i = y_i$ for $t=-0.0919 + 0.0848i$
(which corresponds to $z_3$ in \cite{GalgonKraemerEtAl2014}) and $20$ random right-hand sides $y_i$.
Here, the value for $\rho$ from \eqref{eq:rho} is $\rho = 5$, the largest value in our examples.

Let us note that this example lends itself to a ``shifted'' Krylov subspace approach where systems are
simultaneously solved for various values of $t_j$, but this is less relevant in our context.
\end{example}

\subsection{Stabilization of block \bicg and block \bicgstab}
We first report on a comparison of the block \bicg method,
Algorithm~\ref{alg:bl_BiCG} and block \bicgstab from
\cite{Jbilou05} with the versions which improve stability using a
QR-factorization of the block residuals. For block \bicg, this version
is given explicitly in Algorithm~\ref{alg:bl_BiCG_rQ}.

\begin{figure}
\includegraphics [angle=0,width=.5\textwidth]{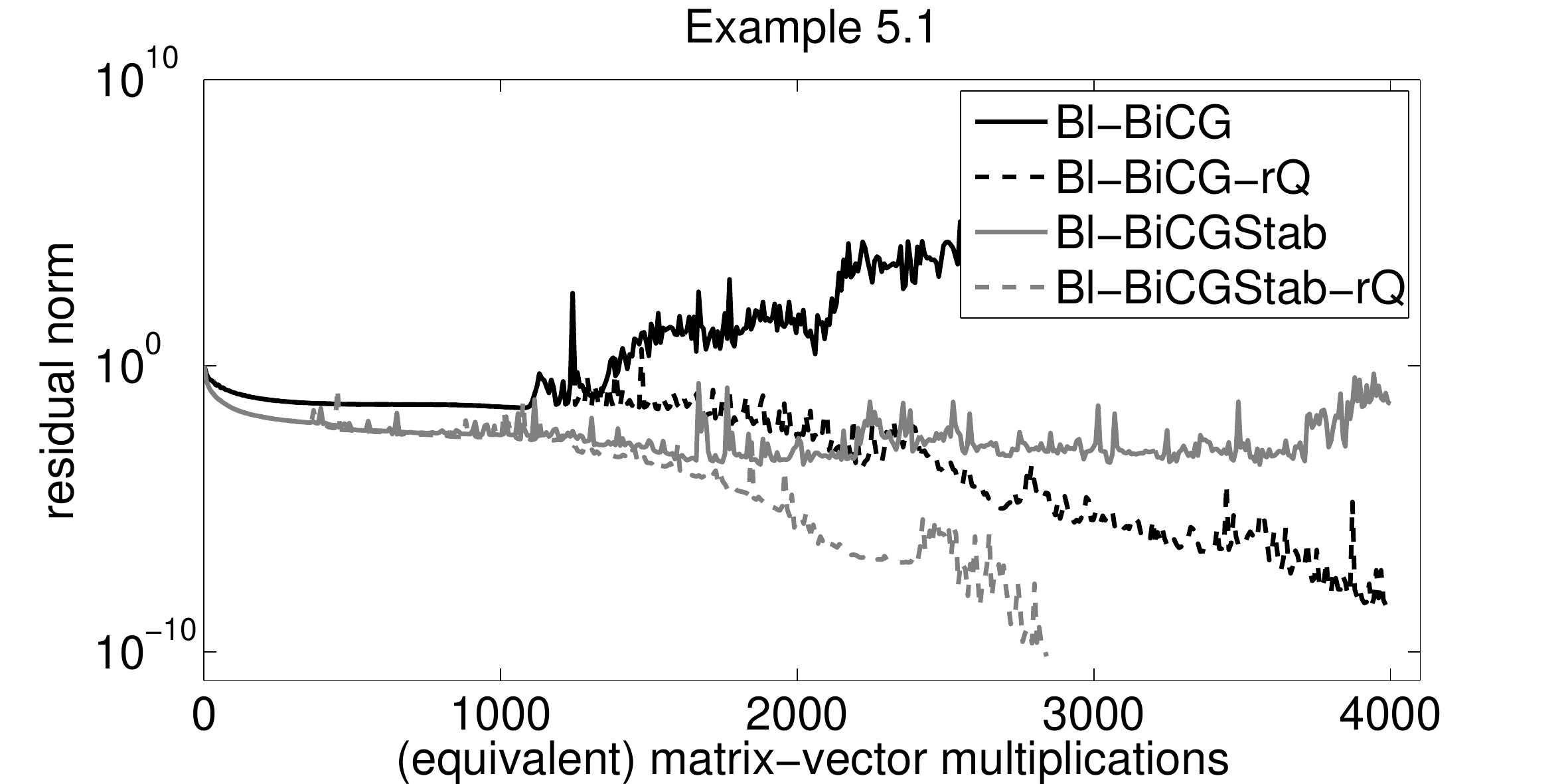}%
\includegraphics [angle=0,width=.5\textwidth]{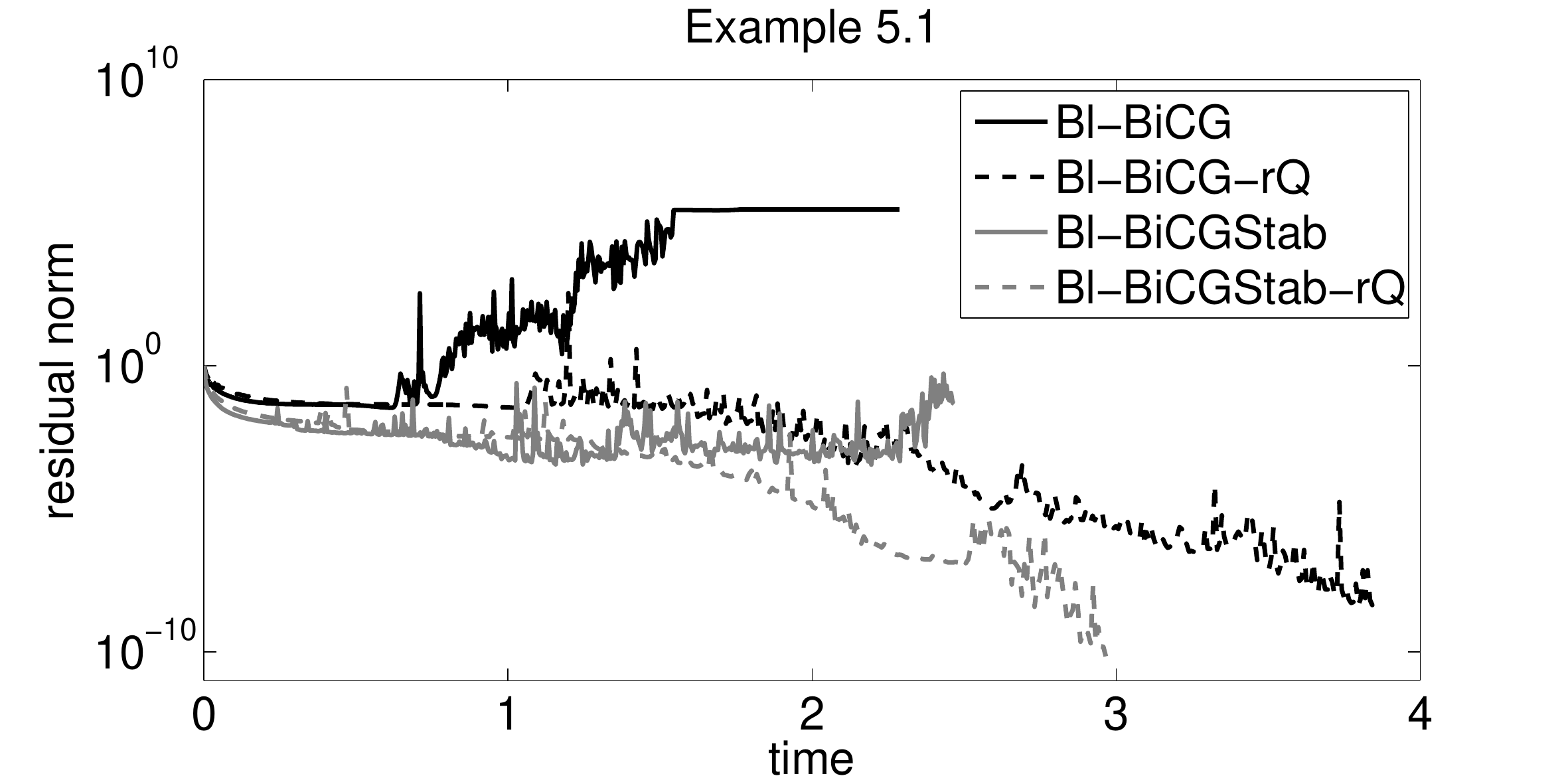}

\vspace*{0.5em}

\includegraphics [angle=0,width=.5\textwidth]{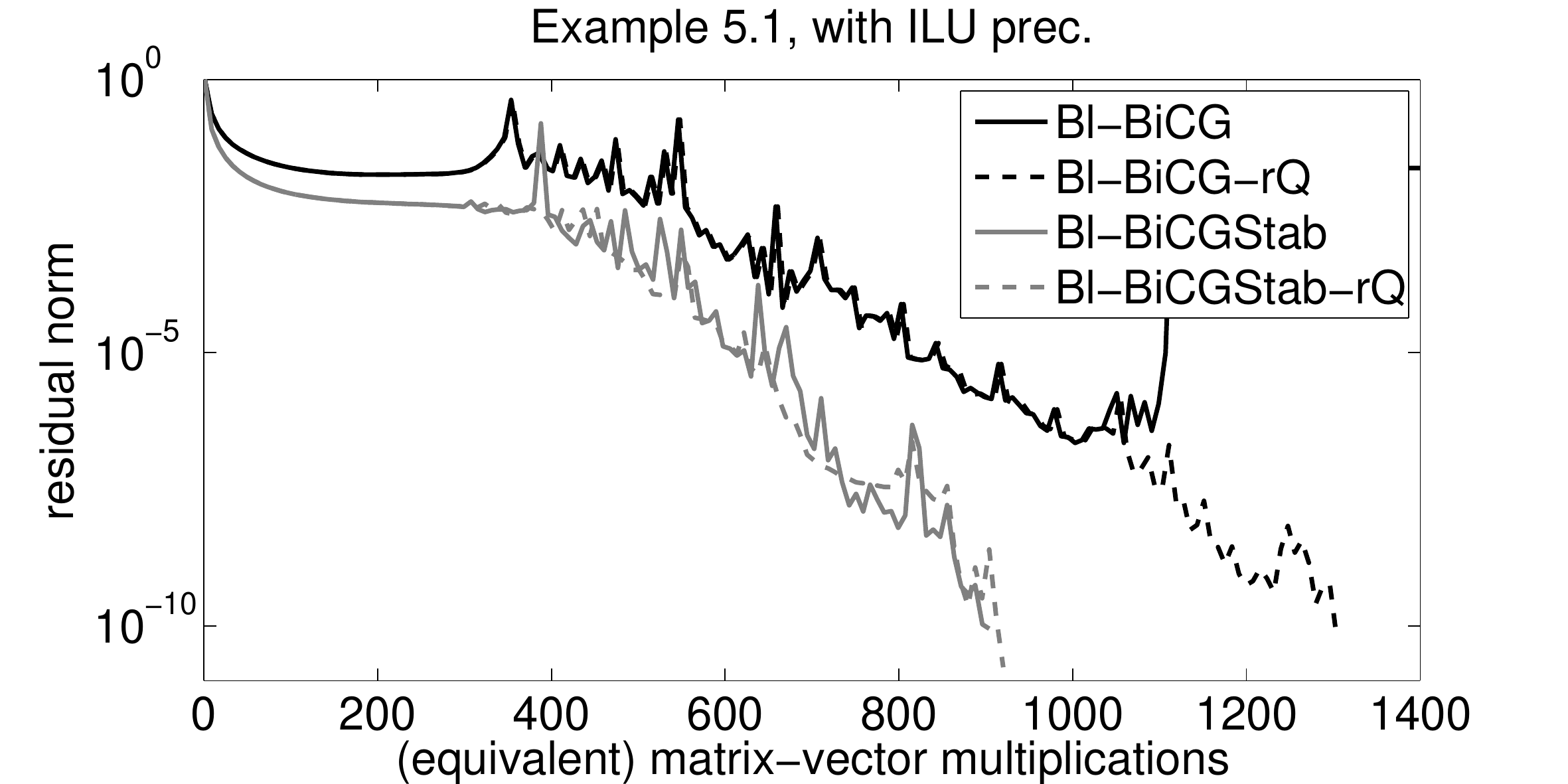}%
\includegraphics [angle=0,width=.5\textwidth]{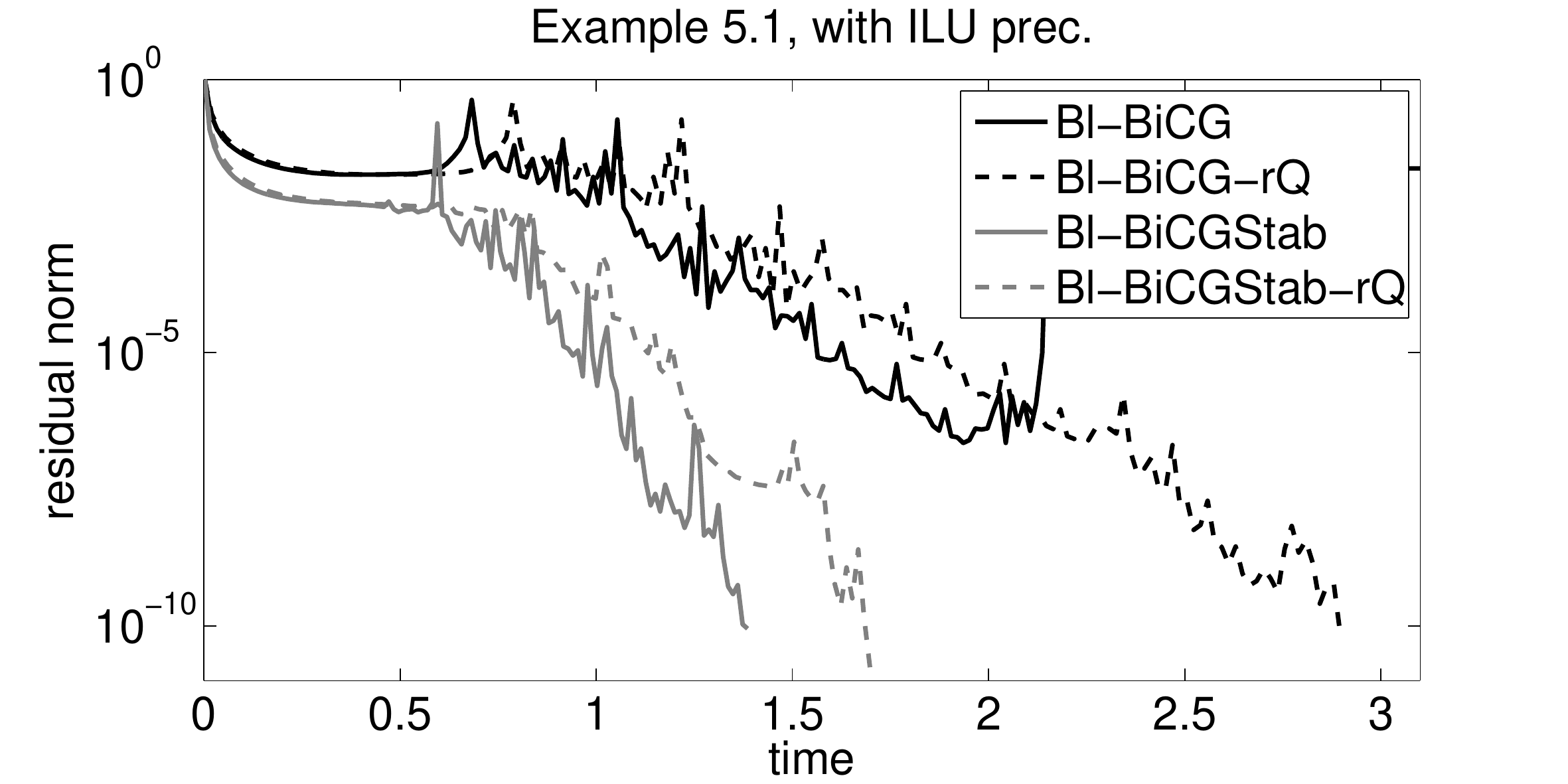}
\caption{Convergence plots for \blbicg and \blbicgstab and their
variants using a QR-factorization of the block residual. Top:
Example~\ref{ex:ex1}, bottom: Example~\ref{ex:ex1} with ILU preconditioning.\label{fig:block_methods}}
\end{figure}

Figure~\ref{fig:block_methods}
shows convergence plots for Example~\ref{ex:ex1} without preconditioning (top row) and
with a (right) no-fill ILU preconditioner (bottom row) obtained via the Matlab function
\texttt{ilu}. The left plots show the
relative Frobenius norm of the block residual as a function of the total number
of matrix-vector multiplications (there are $2s$ of those per iteration). The
right plots show the same residual norms, but now as a function of wall clock
time.
Figure~\ref{fig:block_methods} shows that both, \blbicg as well as
\blbicgstab, can significantly improve when the QR factorizations are used,
sometimes making iterations converge which otherwise would diverge. In the
presence of preconditioning, where the preconditioned matrix is better
conditioned and convergence is reached after a relatively small number of
iterations, QR factorizations has a less pronounced effect, and standard \blbicgstab
actually now converges, too. Since computing a QR-factorization of a block of $s$ vectors of
length $n$ requires an arithmetic cost of $\mathcal{O}(s^2)$ vector operations, this cost is
substantial unless $\rho$ is really small. From the plots in the right column of
Figure~\ref{fig:block_methods} we indeed see that this additional cost is important and that
more than outweighs the small gains in terms of the number of matrix-vector
multiplications.

\begingroup
\setlength{\tabcolsep}{4pt}
\begin{table}[b!]
\caption{Iteration counts \#it, wall clock times and final residual
norms $\| R \|$ for \blbicg and its variant with QR-factorization of
the block residual (top two rows), and their counterparts for
\blbicgstab (bottom two rows). \label{tab:block_methods} }
\begin{footnotesize}
\begin{center}
\addtolength\tabcolsep{-1.7pt}
\begin{tabular}{|lll|lll|lll|lll|lll|}
 \hline
 \multicolumn{15}{|c|}{\em variants of \blbicg (first two rows) and \blbicgstab (last two rows), \bf{no precond.\ }} \\ \hline \hline
  \multicolumn{3}{|c}    {Example \ref{ex:ex1}}  &\multicolumn{3}{|c}  {Example \ref{ex:ex2}}    & \multicolumn{3}{|c}
  {Example \ref{ex:ex3}}& \multicolumn{3}{|c}    {Example \ref{ex:ex4}} & \multicolumn{3}{|c|}    {Example \ref{ex:ex5}}
  \\
  \#it & time  & $\|R\|$ & \#it & time   & $\|R\|$ & \#it & time   & $\|R\|$ & \#it & time   & $\|R\|$ & \#it & time   & $\|R\|$ \\ \hline
  500  & 2.28s & div.\    & 500  & 43.83s & div.\  & 500  & 44.99s & 0.28    & 500  & 62.92s & 1e-03   & 528  & 36.49s & 3e-10   \\
  500  & 3.85s & 1e-08    & 500  & 101.1s & div.\  & 156  & 31.66s & 2e-10   & 237  & 38.85s & 3e-10   & 533  & 79.86s & 3e-10   \\ \hline
  500  & 2.46s & 0.093    & 500  & 45.15s & div.\  & 111  & 9.73s  & 2e-10   & 175  & 20.66s & 2e-10   & 1200 & 77.72s & 2e-3    \\
  355  & 2.96s & 1e-10    & 500  & 91.17s & div.\  & 101  & 18.10s & 2e-10   & 157  & 25.07s & 2e-10   & 957  & 131.7s & 3e-10   \\\hline
  \multicolumn{15}{c}{}\\\hline
  \multicolumn{15}{|c|}{\em variants of Bl-BiCG (first two rows) and Bl-BiCGStab (last two rows), \bf{ILU precond.\  } } \\ \hline \hline
  \multicolumn{3}{|c}    {Example \ref{ex:ex1}}  &\multicolumn{3}{|c}  {Example \ref{ex:ex2}}    & \multicolumn{3}{|c}
  {Example \ref{ex:ex3}}& \multicolumn{3}{|c}    {Example \ref{ex:ex4}} & \multicolumn{3}{|c|}    {\mbox{}} \\
  \#it &  time & $\|R\|$ & \#it &  time & $\|R\|$ & \#it &  time & $\|R\|$ & \#it &  time & $\|R\|$  &  &  & \\ \hline
   500 & 7.72s & 0.04  & 28 & 6.13s & 2e-10 & 50 & 11.35s & 3e-10 & 68  & 23.33s & 2e-10 &&&  \\
   163 & 2.89s & 1e-10 & 28 & 9.18s & 8e-11 & 50 & 17.06s & 3e-10 & 62  & 23.72s & 3e-10 &&& \\ \hline
   113 & 1.38s & 3e-10 & 16 & 2.74s & 4e-11 & 33 & 5.87s  & 6e-11 & 41  & 10.15s & 2e-10 &&&\\
   115 & 1.69s & 3e-11 & 16 & 4.09s & 3e-11 & 33 & 8.84s  & 5e-11 & 40  & 11.72s & 7e-11 &&&\\ \hline
\end{tabular}
\addtolength\tabcolsep{1.7pt}
\end{center}
\end{footnotesize}
\end{table}
\endgroup

Table~\ref{tab:block_methods} summarizes the results
for all our examples. For all four methods, it reports the number of
matrix-vector multiplications and the time needed to reach the stopping
criterion (reduction of the initial block residual by a factor of $10^{-10}$ or
maximum of $500$ iterations reached in Examples~\ref{ex:ex1}-\ref{ex:ex4}, $1,\!200$ for Example~\ref{ex:ex5}). 
We also report the final relative norm
of the block residual which was explicitly re-computed upon termination.
If this final residual norm is larger than 1, we interpret this as divergence, noted as ``div.'' in the table.
Smaller residual norms like $10^{-4}$ may
be interpreted as an indicator of slow convergence, too slow to be competitive
within our setting.

Table~\ref{tab:block_methods} confirms our conclusions from
the discussion of Figure~\ref{fig:block_methods}: QR-factoriz\-ation improves
numerical stability, and it has the potential to turn otherwise
divergent iterations into convergent ones. Its (relative) additional cost is
indicated by the value of $\rho$ from \eqref{eq:rho}, and it is 
relatively small for Examples~\ref{ex:ex1} and \ref{ex:ex4}, whereas it is relatively high for 
the other examples where it thus does not pay off in the cases where the non-stabilized
method is already convergent. The known convergence problems of BiCGStab for 
Example~\ref{ex:ex2}, see \cite{Sle93}, carry over to both variants of the block method, 
and we may conclude that the non-convergence is not a matter of
numerical stability but of \bicgstab and its block counterpart not being able 
to efficiently accommodate the relevant spectral properties
of the respective matrix; see also the discussion in \cite{Sle93}.

\subsection{Comparison of all methods}

\begin{figure}
\includegraphics [angle=0,width=.5\textwidth]{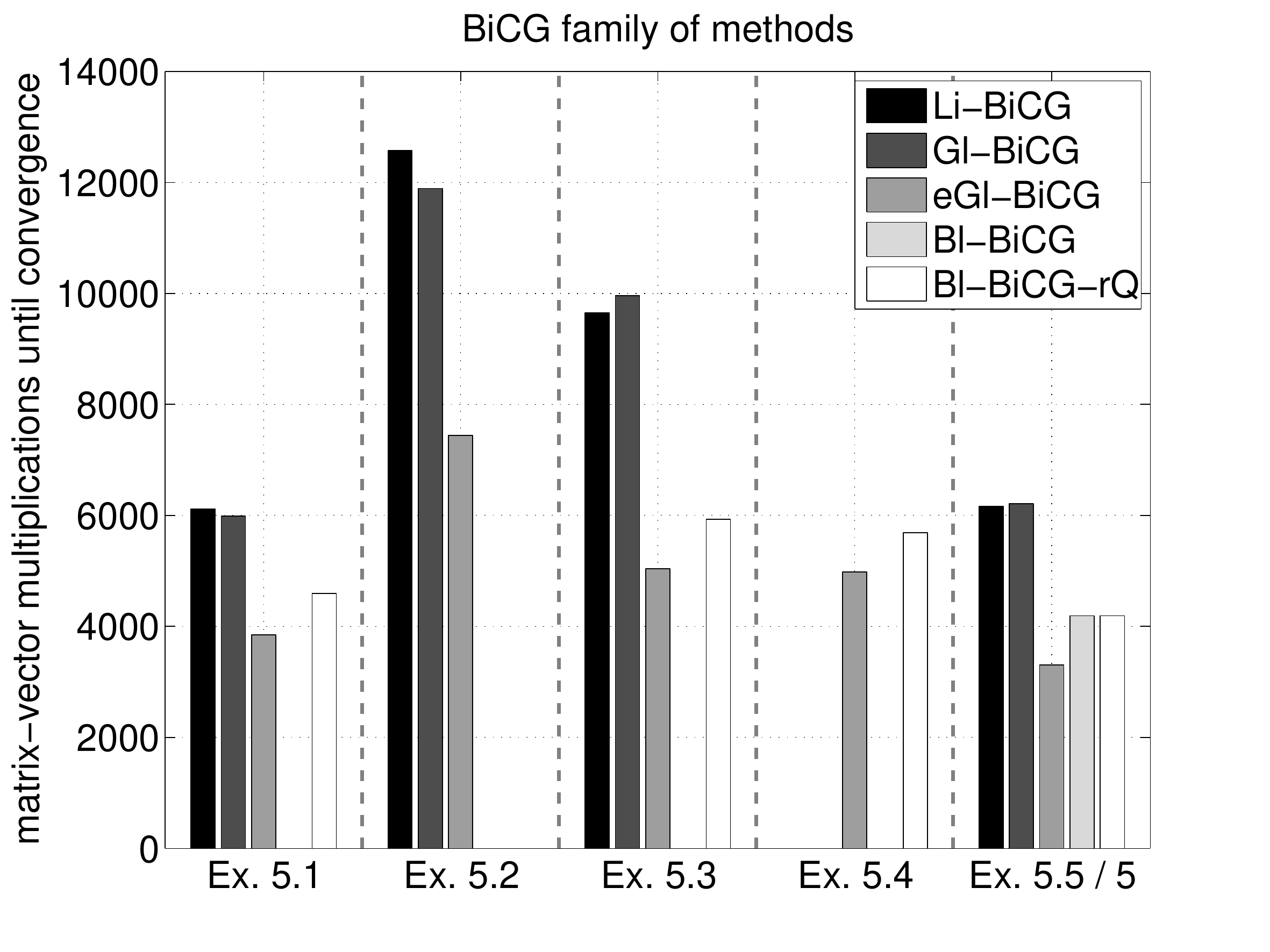}%
\includegraphics [angle=0,width=.5\textwidth]{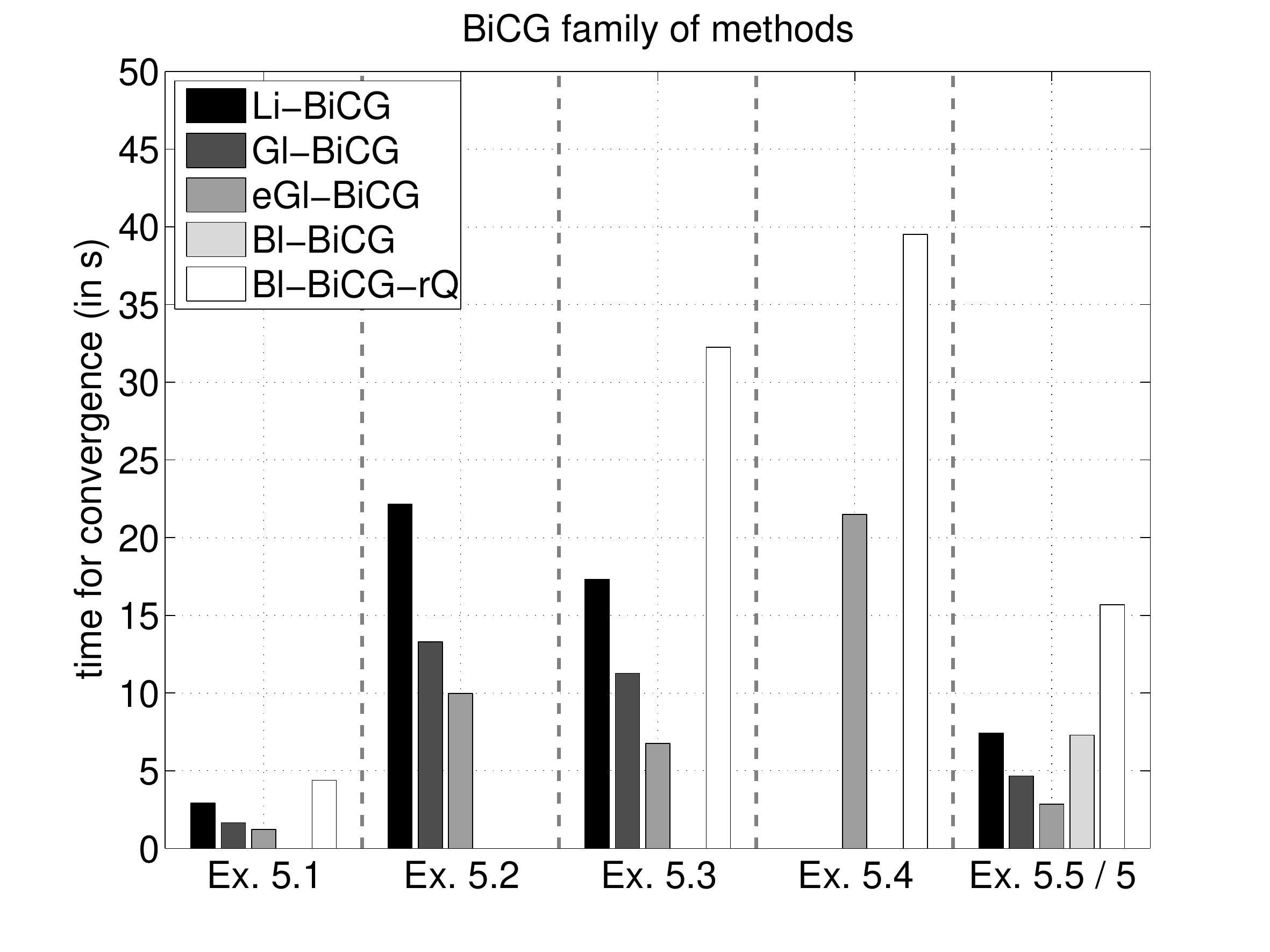}

\includegraphics [angle=0,width=.5\textwidth]{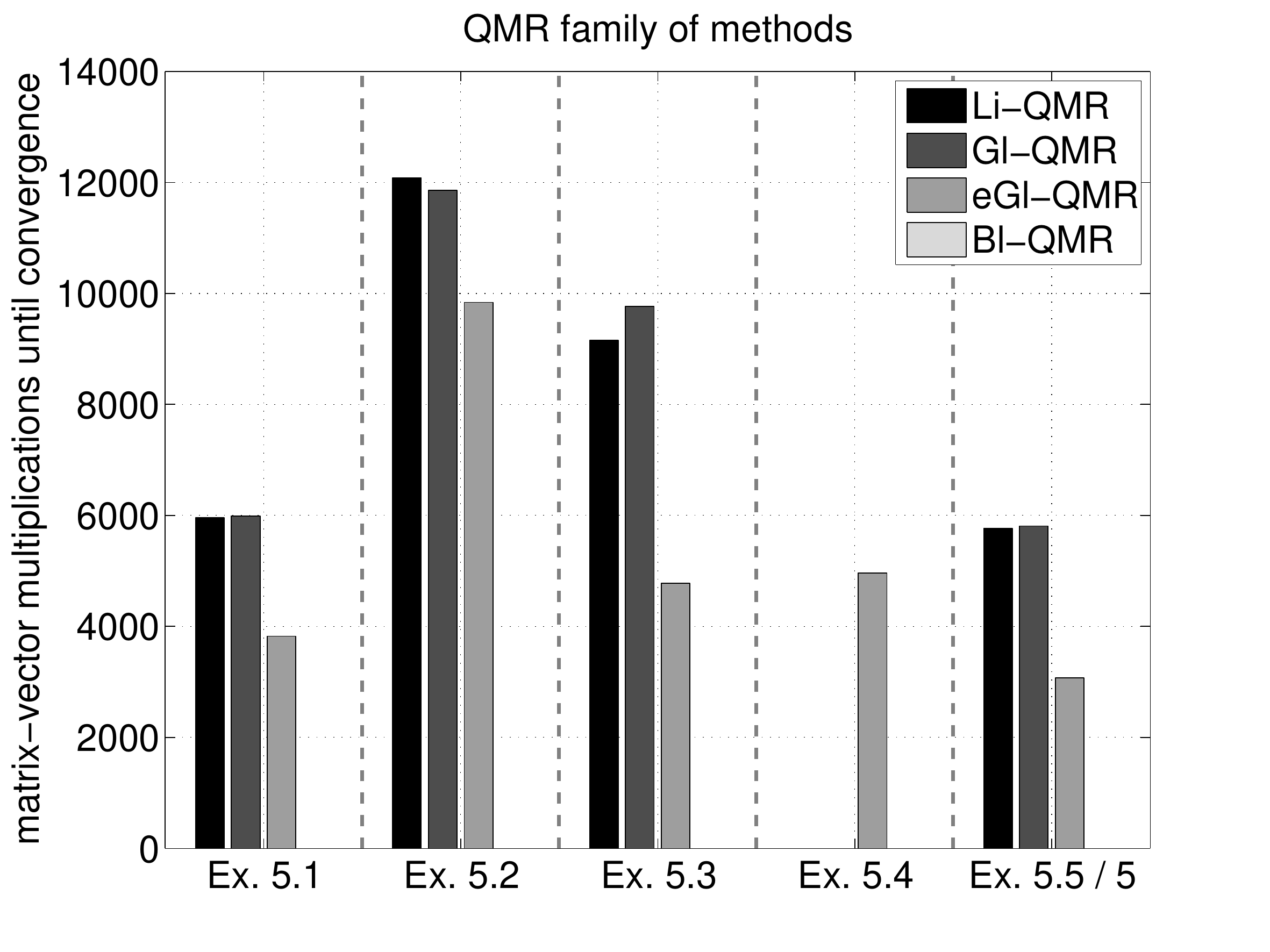}%
\includegraphics [angle=0,width=.5\textwidth]{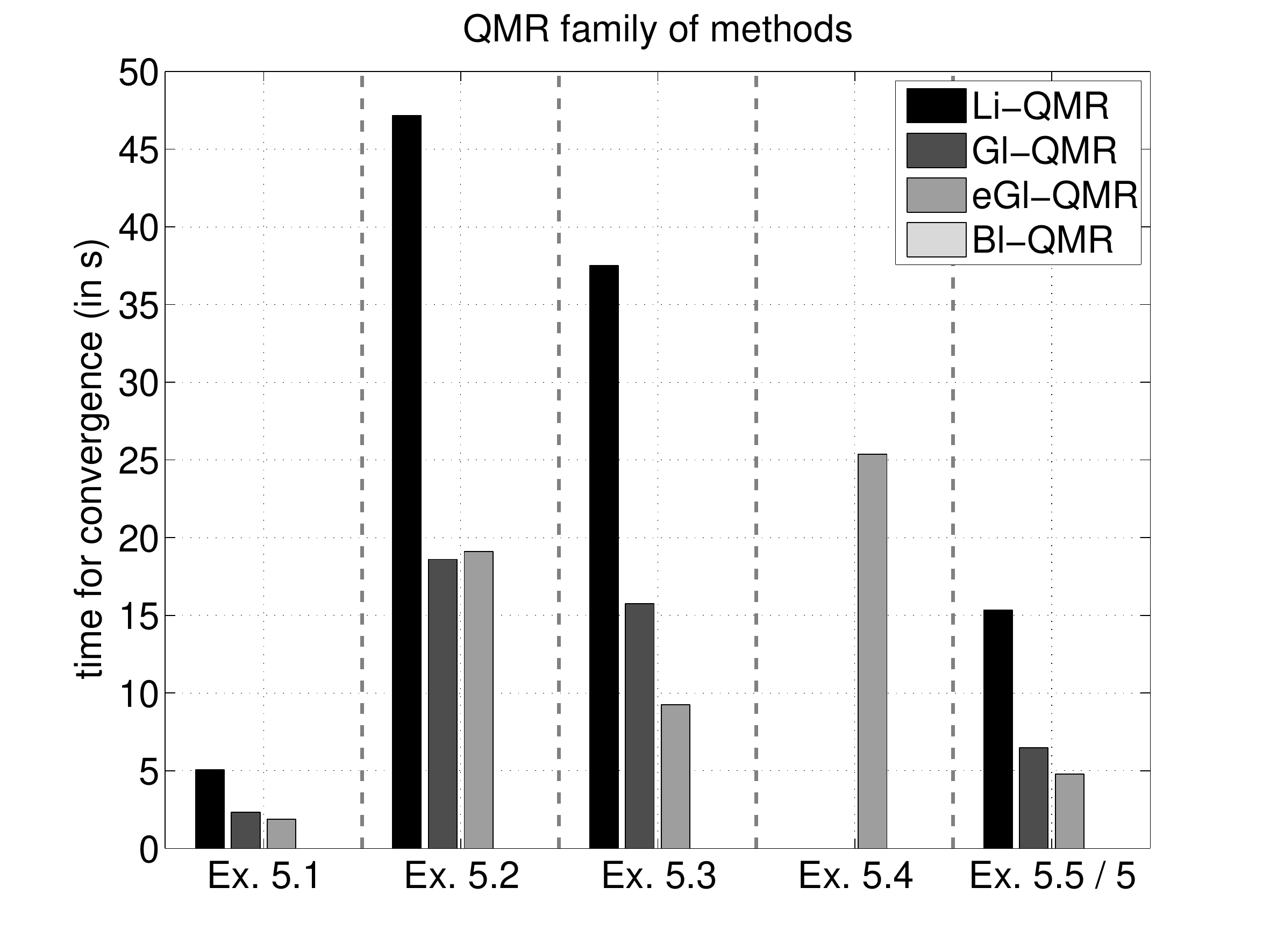}

\includegraphics [angle=0,width=.5\textwidth]{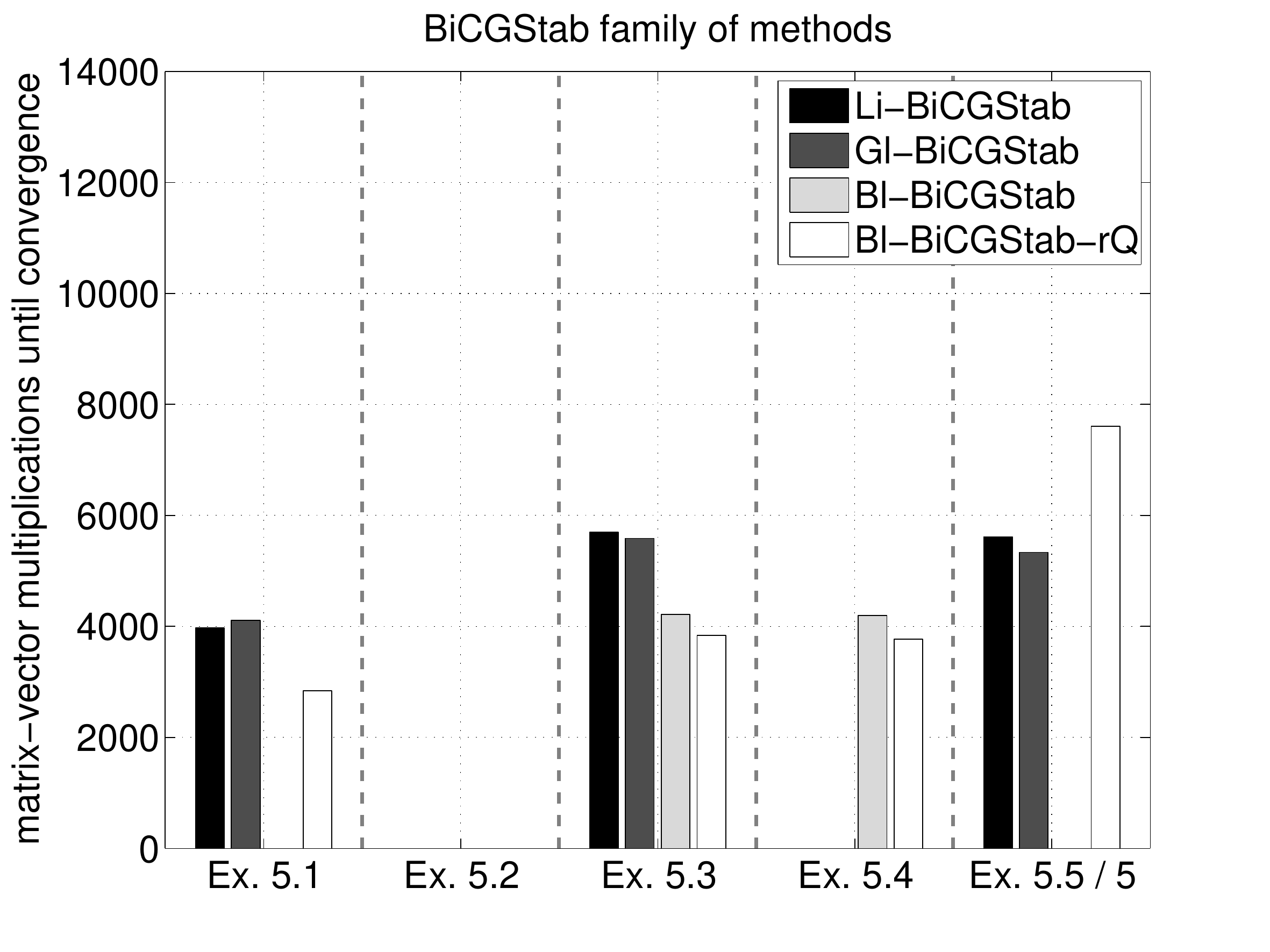}%
\includegraphics [angle=0,width=.5\textwidth]{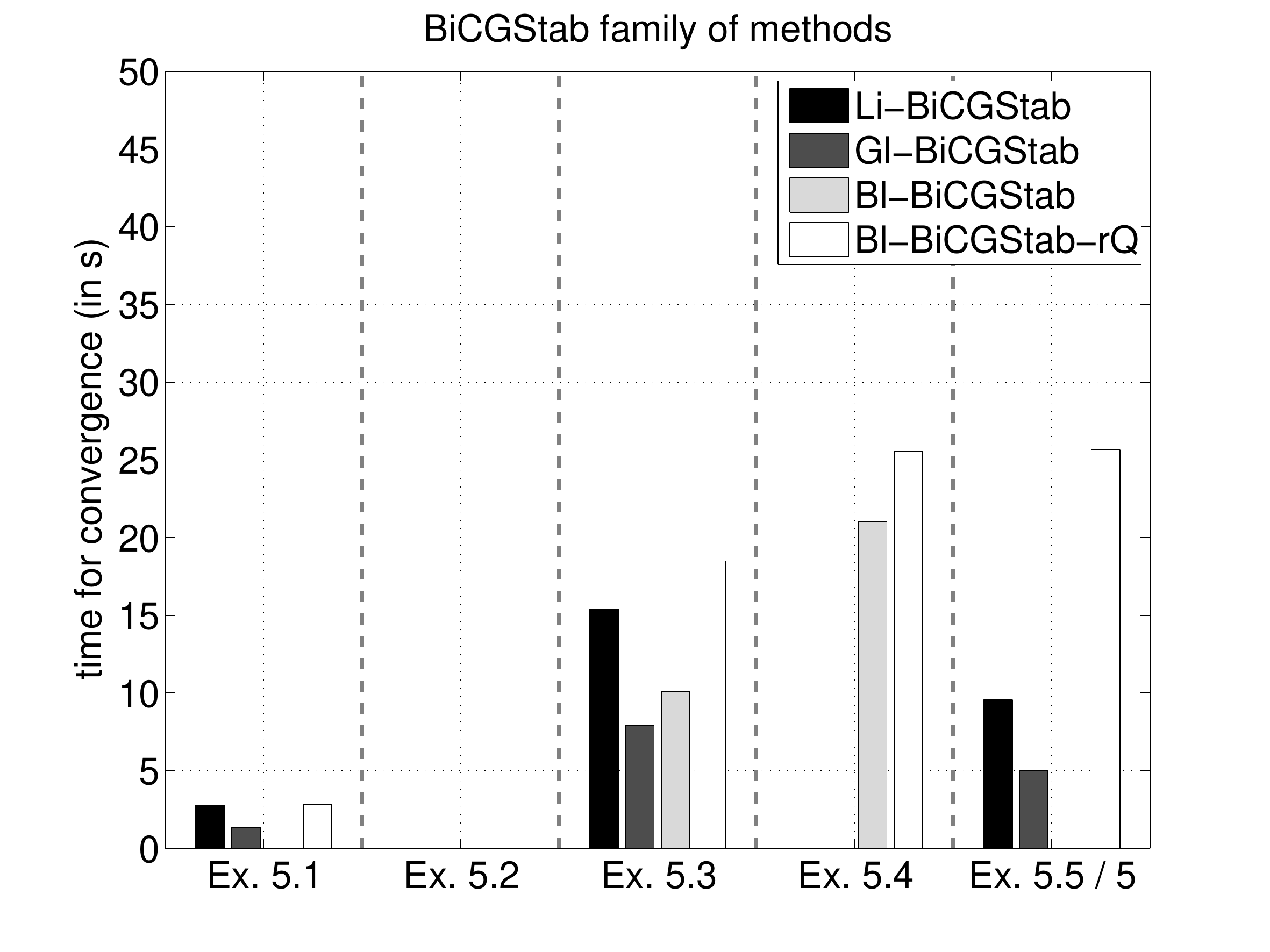}
\caption{Number of matrix-vector multiplications and time for all methods and all examples, no preconditioning.
\label{fig:all_methods_noprec}}
\end{figure}

We proceed by comparing the number of matrix-vector multiplications and the wall clock times for all methods.
Figure~\ref{fig:all_methods_noprec} does so for the case without preconditioning. We use bar plots to allow for an immediate visual comparison
of the methods and we group our measurements by ``families'' of methods. Since Example~\ref{ex:ex5} needs much more iterations than the other examples,
we scaled the bar plots for this example by dividing by 5.
A missing vertical bar indicates that the corresponding method either diverged or stagnated at a large residual norm.
Instead of implementing a \blqmr method by ourselves we used a Matlab implementation by Freund and Malhotra which was publicly available at least until 2008 at the web site of
Lucent technologies. This implementation
is very cautious about possible ill-conditioning and (near) deflation, and for this reason it actually does all matrix-vector products one at a time, i.e.\
it does not use matrix-block vector products. As a result, this block QMR implementation can not be competitive with respect to timings, and for the unpreconditioned
systems considered here we obtained premature break downs in all examples. The situation becomes slightly different in the preconditioned case, see below.

We can make the following observations: The loop-interchanged and the global methods require almost the same number of matrix-vector 
multiplications in all examples. Since the other arithmetic work is also comparable, this should result in similar wall clock times. This is, 
however, not what we see, the loop interchanged methods requiring substantially more time. We attribute this to the fact that the loop 
interchanged methods are the only ones were there is a non-negligible amount of interpreted Matlab code. ``Plain'' block methods suffer from 
non-con\-ver\-gen\-ce in a significant number of cases (8 out of 10). Using a QR-factorization of the residuals reduces the number of 
failures to 2. When they work, the block methods reduce the number of matrix-vector multiplications as compared to loop-interchanged or 
global methods by never more than a factor of 2. Because of the additional arithmetic cost, block methods never win in terms of wall clock 
time, and for larger values of $\rho$ (Examples~\ref{ex:ex2}, \ref{ex:ex3} and \ref{ex:ex5}) they require more than twice the time than the 
corresponding global method across all ``families''.     
The economic versions of the global methods always appear as the best within their respective ``family''. They do not exhibit convergence
problems, and they reduce the number of matrix-vector multiplications when compared with the global and loop interchanged methods. At the same time, their wall clock times are smallest. Overall, the economic global BiCG method appears to be the most efficient one, in tie with global BiCGStab for Example~\ref{ex:ex1} and block BiCGStab for Example \ref{ex:ex4}. Note that these are the examples with the smallest value for $\rho$.

\begin{figure}
\includegraphics [angle=0,width=.5\textwidth]{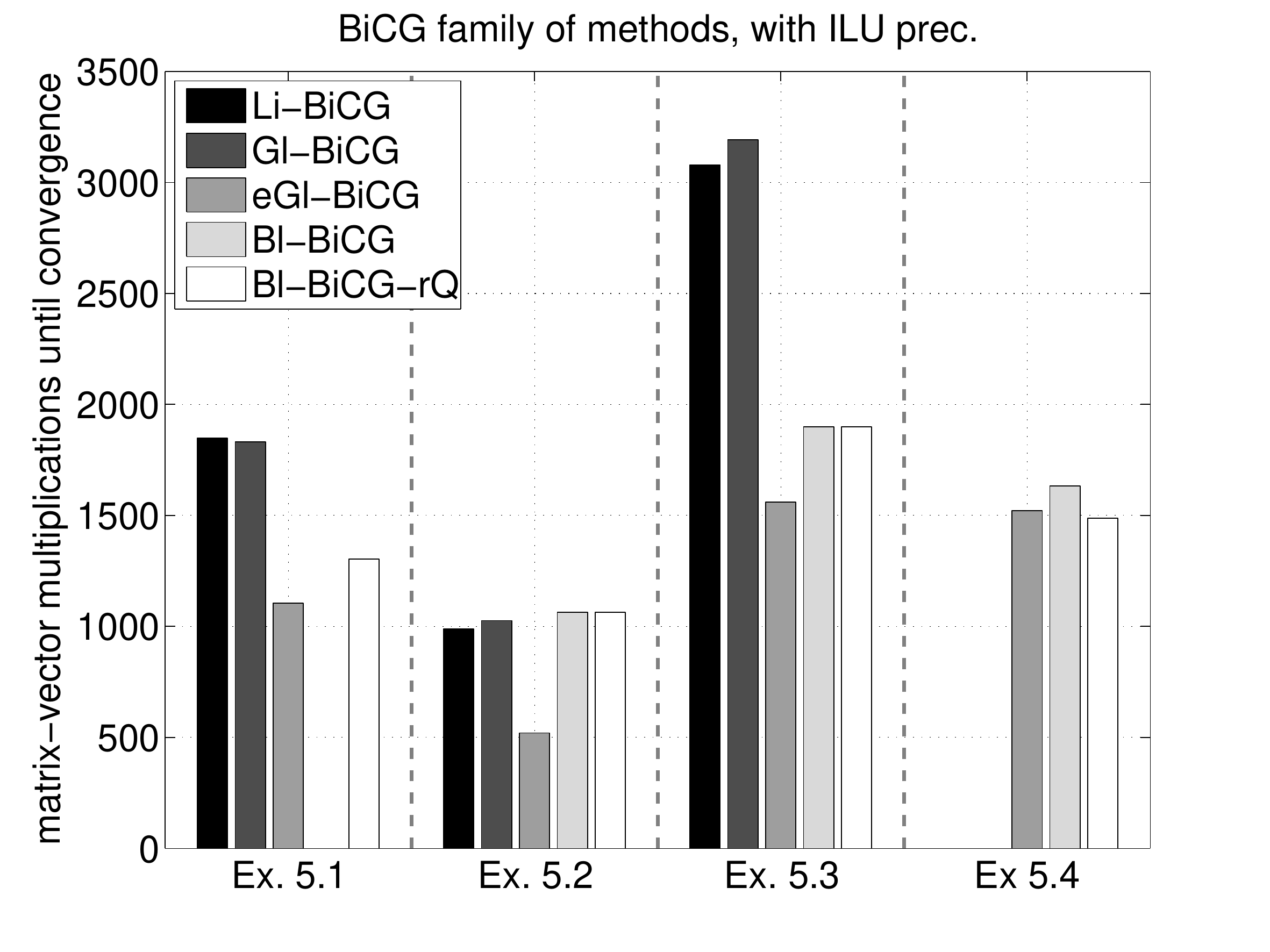}%
\includegraphics [angle=0,width=.5\textwidth]{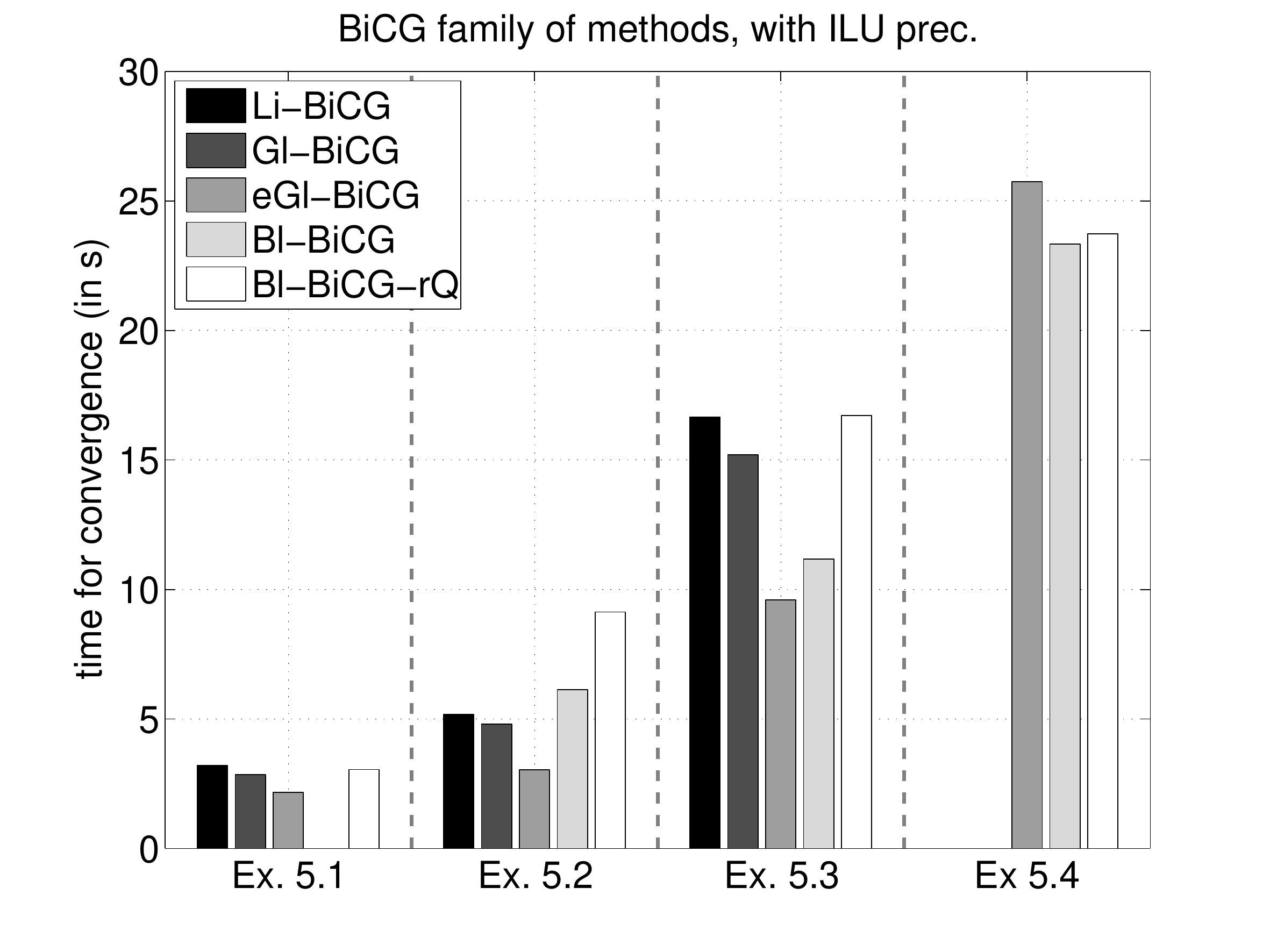}

\includegraphics [angle=0,width=.5\textwidth]{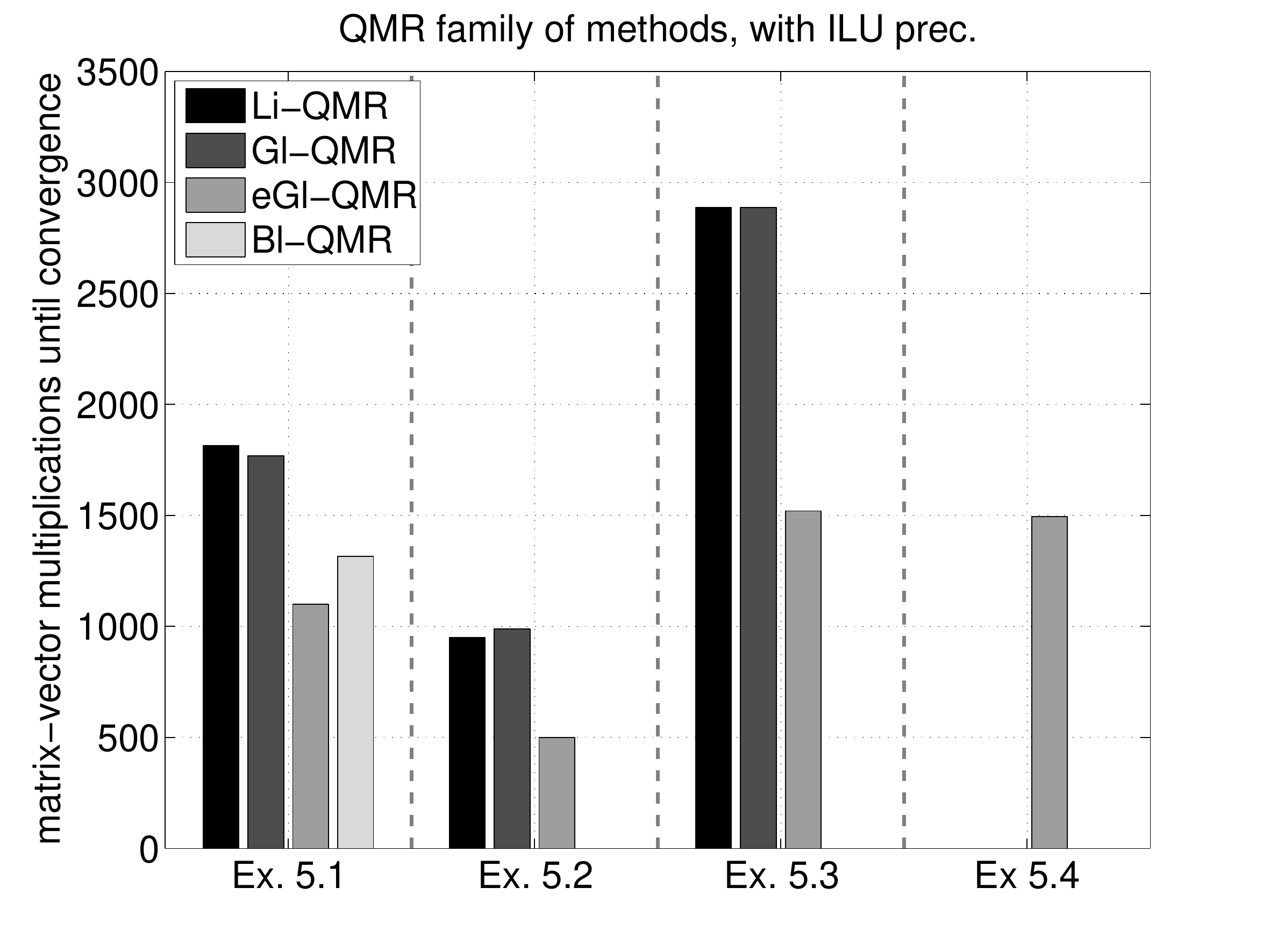}%
\includegraphics [angle=0,width=.5\textwidth]{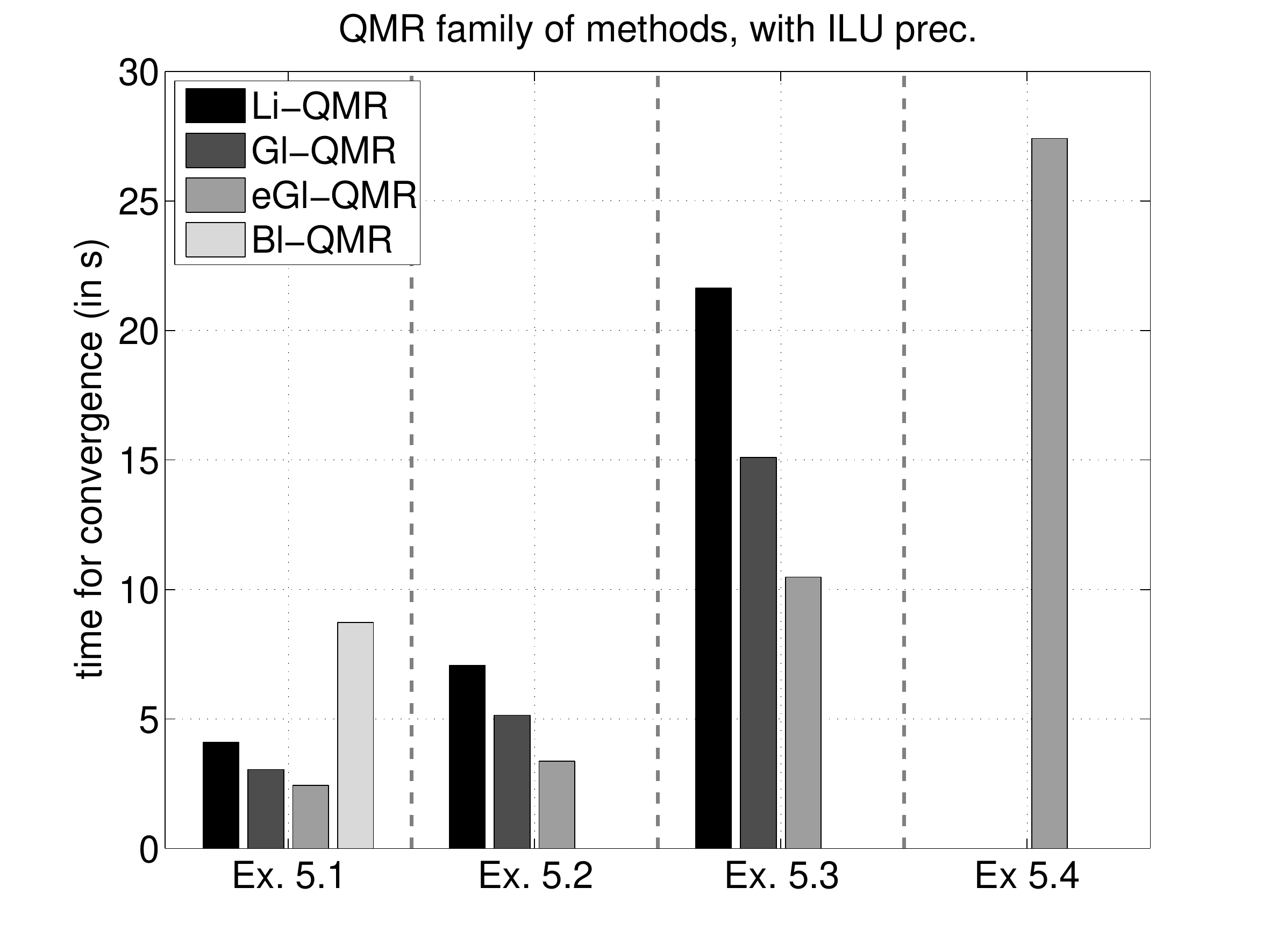}

\includegraphics [angle=0,width=.5\textwidth]{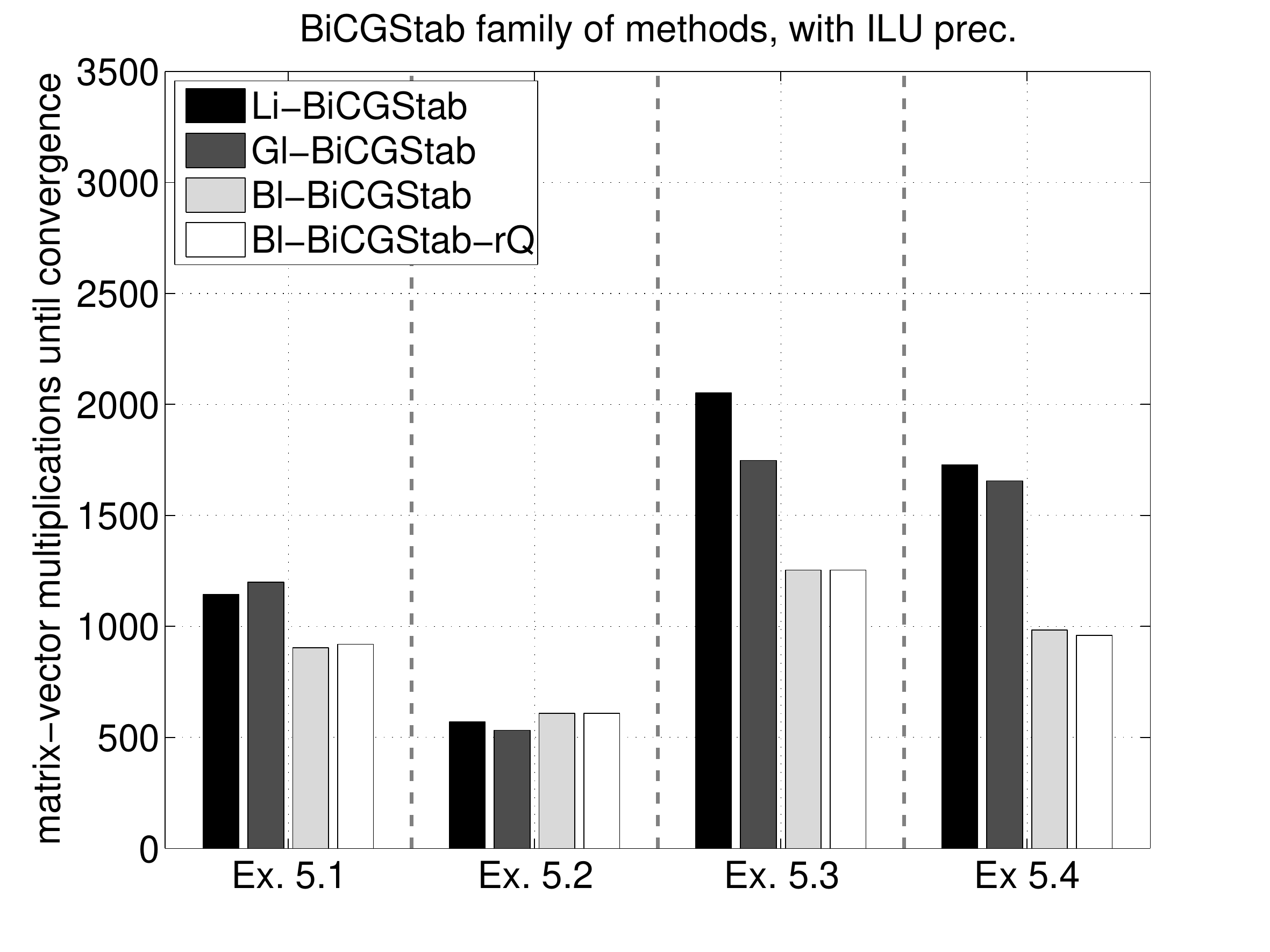}%
\includegraphics [angle=0,width=.5\textwidth]{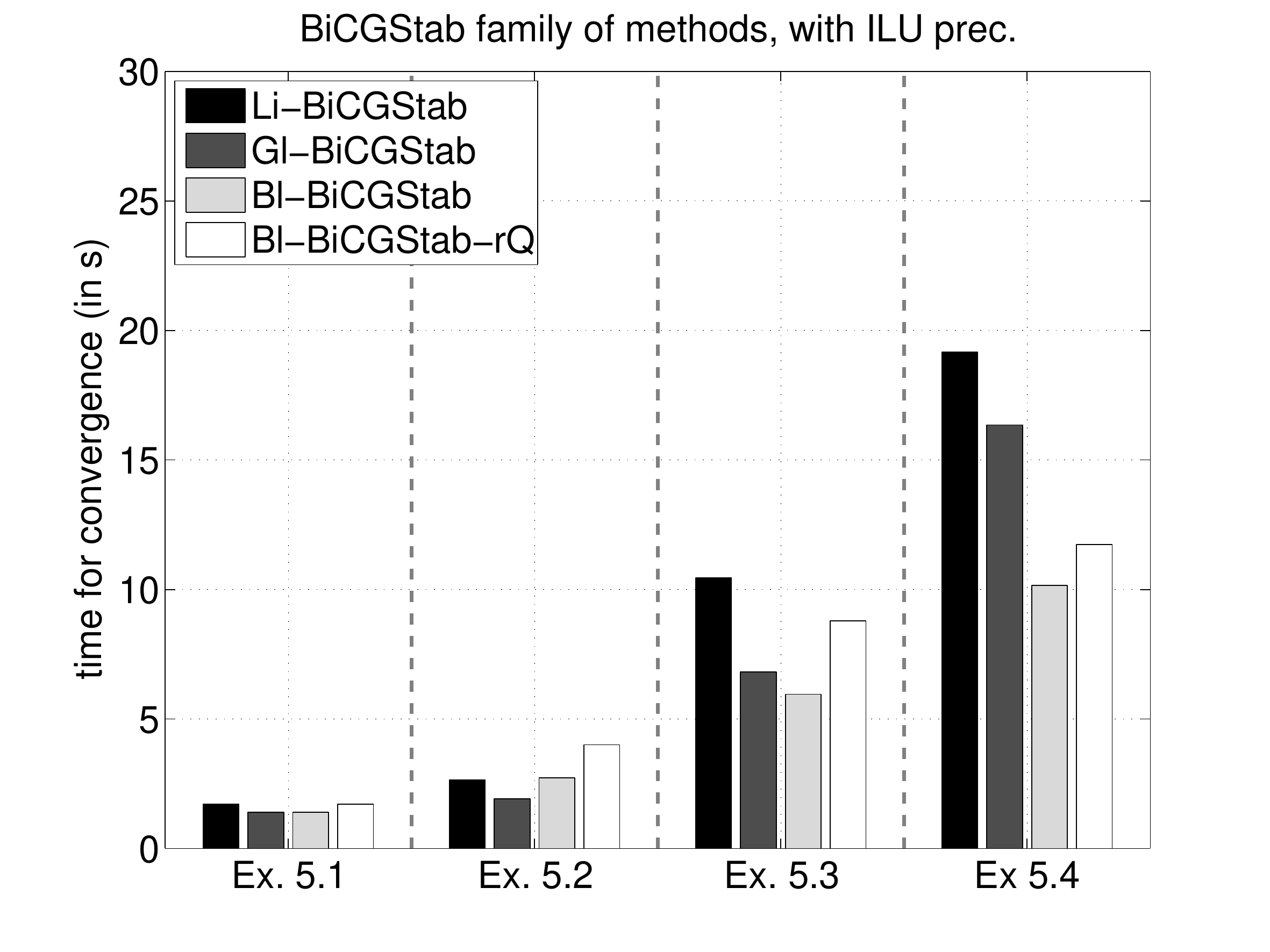}
\caption{Number of matrix-vector multiplications and time for all methods, Examples~\ref{ex:ex1}-\ref{ex:ex4} with ILU preconditioning.
\label{fig:all_methods_prec}}
\end{figure}

Figure~\ref{fig:all_methods_prec} shows the results for Examples~\ref{ex:ex1}-\ref{ex:ex4} where we now use ILU (right) preconditioning. 
More precisely,
we used Matlab's \texttt{ilu} do obtain a no-fill ILU in Examples~\ref{ex:ex1}-\ref{ex:ex3} and an ILU  with threshold and pivoting with a 
drop tolerance of $5 \cdot 10^{-2}$ in
Example~\ref{ex:ex4}. The number of matrix-vector multiplications and the wall clock times decrease for all methods as compared
to the un-preconditioned case.
Block QMR now converges for Example~\ref{ex:ex1}, but its wall-clock time is not competitive for the reasons explained earlier. The other block methods fail only once for BiCG, and never for BiCGStab. While the economic versions of the global methods still yield the best timings within their ``family'', the block BiCGStab methods and global BiCGStab now perform better in terms of wall clock time for all four examples. The value of $\rho$ being particularly small for Example~\ref{ex:ex4}, this is the example where block BiCGStab gains most against loop interchanged or 
global BiCGStab.

Summarizing our findings we can conclude that one should use the block methods with particular care since there is a danger of non-
convergence due to numerical instabilities.
This can be somewhat attenuated by using the suggested QR-factorization of the block residuals. The additional arithmetic cost in the block 
methods should not be neglected, and---depending on $\rho$---
there must be a substantial gain in matrix-vector multiplications in the block methods if this additional cost is to be outweighed. 
Global methods and loop interchanged methods require about the same amount of matrix-vector multiplications and additional arithmetic cost,
so that they should require about the same time, too, if it were not for special effects in a Matlab implementation which mixes compiled 
and interpreted code. The economic versions of global BiCG and global QMR appear to perform well with respect to both, stability and 
efficiency. For better conditioned systems, e.g.\ when an efficient preconditioner is at hand,
the block BiCGStab methods and global BiCGStab behave stably, and then their runtime is less than for the economic global methods.


\bibliography{reference}

\end{document}